\documentclass[12pt]{article}
\usepackage{amsthm,amssymb,amsmath,graphicx,cite}
\usepackage{algpseudocode,algorithm,algorithmicx}

\usepackage[margin = 1.2in]{geometry}

\usepackage{color}

\newcommand{\relint}{\mathsf{ri}}
\newcommand{\relbdy}{\mathsf{rbd}}

\newcommand{\diam}{\mathsf{diam}}

\newcommand{\faces}{\mathsf{faces}}
\newcommand{\vertices}{\mathsf{vertices}}
\newcommand{\Oh}{{\mathcal O}}
\newcommand{\dmin}{\displaystyle\min}

\newcommand{\1}{\mathbf{1}}

\newcommand{\ip}[2]{\left\langle #1 , #2 \right\rangle}    
\newcommand{\R}{{\mathbb R}}
\newcommand{\T}{{\mathcal T}}
\newcommand{\B}{{\mathbb B}}

\newcommand{\conv}{{\mathsf{conv}}}

\newcommand{\dist}{\mathsf{dist}}

\newcommand{\dom}{{\mathrm{dom}}}

\newcommand{\dargmin}{\displaystyle \argmin}

\DeclareMathOperator*{\argmin}{argmin}
\DeclareMathOperator*{\argmax}{argmax}

\DeclareMathOperator{\lspan}{span}

\newtheorem{lemma}{Lemma}
\newtheorem{theorem}{Theorem}
\newtheorem{proposition}{Proposition}

\newtheorem{corollary}{Corollary}

\newtheorem{example}{Example}
\newtheorem{definition}{Definition}
\newtheorem{assumption}{Assumption}
\def\transp{^{\text{\sf T}}}
\newcommand{\matr}[1]{\begin{bmatrix} #1 \end{bmatrix}}    

\title{The condition number of a function relative to a set
} 
\author{
David H. Gutman\thanks{Department of Industrial, Manufacturing, and Systems Engineering, Texas Tech University, USA, {\tt david.gutman@ttu.edu}}
\and
Javier F. Pe\~na\thanks{Tepper School of Business,
Carnegie Mellon University, USA, {\tt jfp@andrew.cmu.edu}}
}

\begin{document}

\maketitle

\begin{abstract}
The {\em condition number} of a differentiable convex function, namely the ratio of its smoothness to strong convexity constants,  is closely tied to fundamental properties of the function.  In particular, the condition number of a quadratic convex function is the square of the aspect ratio of a canonical ellipsoid associated to the function.  Furthermore, the condition number of a function bounds the linear rate of convergence of the gradient descent algorithm for unconstrained convex minimization. 

We propose a condition number of a differentiable convex function relative to a reference convex set and distance function pair.  This relative condition number is defined as the ratio of {\em relative smoothness} to {\em relative strong convexity} constants.  We show that the relative condition number extends the main properties of the traditional condition number both in terms of its geometric insight and in terms of its role in characterizing the linear convergence of first-order methods for constrained convex minimization.

When the reference set $X$ is a convex cone or a polyhedron and the function $f$ is of the form $f = g\circ A$, we provide characterizations of and bounds on the condition number of $f$ relative to $X$ in terms of the usual condition number of $g$ and a suitable condition number of the pair $(A,X)$.

\end{abstract}


%
%
\section{Introduction}
\label{sec.intro}

Let $f:\R^m\to\R\cup\{\infty\}$ be a convex differentiable function.  The {\em condition number} of $f$ is the ratio $L_f/\mu_f$ where  $L_f$ and $\mu_f$ are  respectively the {\em smoothness} and {\em strong convexity} constants of the function $f$.  See Definition~\ref{def.regular} and equation~\eqref{eq.classic.constants} below.  The condition number $L_f/\mu_f$ is closely tied to a number of fundamental properties of the function $f$.  In the special case when $f$ is a quadratic convex function the condition number has the following geometric interpretation.  Suppose $f(x) = \frac{1}{2}\|Ax-b\|_2^2$ 
where $A\in \R^{n\times n}$ is non-singular.  Then the condition number of $f$ is 
\begin{equation}\label{eq.classic.cn}
\frac{L_f}{\mu_f} = \|A\transp A\| \cdot \|(A\transp A)^{-1}\| = (\|A\|\cdot\|A^{-1}\|)^2.
\end{equation}
The latter quantity is the square of the aspect ratio of the ellipsoid $A(\B^n):= 
\{Ax :x\in \R^n, \|x\|_2\le 1\}$ since $\|A\|$ and $1/\|A^{-1}\|$ are respectively the radius of the smallest ball that contains $A(\B^n)$ and the radius of the largest ball contained in $A(\B^n)$.

The condition number $L_f/\mu_f$ also bounds the linear convergence rate of the gradient descent algorithm for the unconstrained minimization problem 
$$f^\star := \dmin_{x\in\R^m} f(x).$$ 
More precisely, for a suitable choice of step sizes the iterates $x_k, \; k=0,1,\dots$ generated by the gradient descent algorithm satisfy
\[
\|X^\star - x_k\|_2^2 \le \left(1 - \frac{\mu_f}{L_f} \right)^k\|X^\star - x_0\|_2^2
\]
and
\[
f(x_k) - f^\star \le \frac{L_f}{2}\left(1 - \frac{\mu_f}{L_f} \right)^k\|X^\star - x_0\|_2^2,
\]
where $X^\star:=\{x\in \R^n: f(x) = f^\star\}$ and $\|X^\star - x\|_2 = \inf_{y\in X^\star} \|y-x\|_2$.  The  articles~\cite{BubeLS15,DrusFR16,KariV17,MaGV17,NecoNG18,Nest04,Nest13}, among others, discuss the above type of linear convergence and a number of interesting related developments.  In particular, Necoara, Nesterov and Glineur~\cite{NecoNG18} establish linear convergence properties for a wide class of first-order methods under assumptions that are relaxations of strong convexity.

\medskip

Let $f:\R^m\to\R\cup\{\infty\}$ be a convex differentiable function, $X \subseteq \dom(f)$ be a convex set, and $D:X\times X\rightarrow \R_+$ be a {\em distance-like} function, that is, $D(y,x) \ge 0$ and $D(x,x) = 0$ for all $x,y\in X$.  We propose a relative smoothness constant $L_{f,X,D}$ and  a relative strong convexity constant $\mu_{f,X,D}$ of the function $f$ relative to the pair $(X,D)$.  See Definition~\ref{def.relative}  and equation~\eqref{eq.rel.constants} below for details.    
We show that the relative condition number $L_{f,X,D}/\mu_{f,X,D}$ extends the above properties of the traditional condition number $L_f/\mu_f$ both in terms of its geometric insight and in terms of its role in characterizing the linear convergence of first-order methods for the constrained convex minimization problem
\begin{equation}\label{eq.Prob}
f^\star := \dmin_{x\in X} f(x).
\end{equation}
As Example~\ref{example.sing.vals} illustrates, the relative condition number depends on the combination of the constraint set $X$ and the function $f$.  In particular, Example~\ref{example.sing.vals} shows that the relative condition number $L_{f,X,D}/\mu_{f,X,D}$ can be vastly different (both smaller or larger) than the usual condition number $L_f/\mu_f$ depending on how the shape of $X$ fits $f$.  Example~\ref{example.sing.vals} also shows that $\mu_{f,X,D}$ can be strictly positive in cases when $\mu_f = 0$.
Our main results highlight  deeper connections between the relative constants and geometric features of the set $X$.  In particular,  when $f = g\circ A$ for some matrix $A\in \R^{m\times n}$ and $g:\R^m \rightarrow \R\cup\{\infty\}$, and $X$ is conic or polyhedral, we provide characterizations of and bounds on $L_{f,X,D}$ and $\mu_{f,X,D}$ in terms of $L_g$ and $\mu_g$ and some condition properties of the pair $(A,X)$.

We show that the relative condition number $L_{f,X,D}/\mu_{f,X,D}$ and some related quantities readily yield linear convergence rates for the mirror descent, Frank-Wolfe, and Frank-Wolfe with away steps algorithms for the constrained minimization problem~\eqref{eq.Prob}.
We should note that these linear convergence  properties have been previously established in~\cite{BeckT04,BeckS15,LacoJ15,LuFN18,GuelM86,NecoNG18,Nest13,PenaR16,Tebo18} under various kinds of assumptions.   Our approach shows that all of these linear convergence results hinge on a similar type of relative conditioning.  Our approach also reveals that several linear convergence results can be sharpened. We show that the linear convergence of the mirror descent algorithm (Proposition~\ref{prop.lin.mirror} and Proposition~\ref{prop.lin.mirror.2}) holds for a sharper rate and under  more general assumptions than those in~\cite{LuFN18,Tebo18}.  More precisely, Proposition~\ref{prop.lin.mirror} and Proposition~\ref{prop.lin.mirror.2} 
show that linear convergence holds under new conditions of relative quasi-strong convexity and relative functional growth that are typically weaker than the type of relative strong convexity assumed in~\cite{LuFN18,Tebo18}. 
In contrast to the previous results in~\cite{BeckT04,GuelM86}, our linear convergence result for the Frank-Wolfe algorithm (Proposition~\ref{prop.lin.FW})
is stated in terms of an affine invariant relative condition number defined via a natural {\em radial} distance function. Our approach based on the relative condition number yields a proof of linear convergence for the Frank-Wolfe with away steps algorithm that is significantly shorter, simpler, and at least as sharp as or sharper than the ones previously presented in~\cite{BeckS15,LacoJ15,PenaR16}.   Unlike previous approaches, our proof of linear convergence of the Frank-Wolfe with away steps algorithm (Proposition~\ref{prop.lin.fwa}) highlights some similarities with the proof of  linear convergence of the regular Frank-Wolfe algorithm (Proposition~\ref{prop.lin.FW}).  Like the results presented in~\cite[Appendix C and D]{LacoJ15}, the linear convergence of the Frank-Wolfe  with away steps  algorithm (Proposition~\ref{prop.lin.fwa}) is stated in terms of an affine invariant relative condition number. 

The relative constants $L_{f,X,D}$ and $\mu_{f,X,D}$ are defined {\em globally.}  In particular, they do not depend on any specific point in $X$.   We consider several variants of relative strong convexity following the constructions of Necoara, Nesterov and Glineur~\cite{NecoNG18}. In particular, we define a  {\em relative quasi-strong convexity constant} $\mu_{f,X,D}^\star$ and a {\em    relative functional growth constant} $\mu_{f,X,D}^\sharp$.  See Definition~\ref{def.qrelative} and equation~\eqref{eq.qrel.const}.  Unlike $\mu_{f,X,D}$, the constants $\mu_{f,X,D}^\star$ and $\mu_{f,X,D}^\sharp$  depend on the set of minimizers $X^\star$ of $f$ on $X$.  We show that relative quasi-strong convexity is a relaxation of relative strong convexity.  We also show that under suitable assumptions relative functional growth is a relaxation of relative quasi-strong convexity.  Not surprisingly, there are classes of non-strongly convex functions for which the constant $\mu_{f,X,D}^\sharp$ is positive while $\mu_{f,X,D}$ and $\mu_{f,X,D}^\star$ may not be.  (See Theorem~\ref{thm.main.growth}.)

Our work draws on and connects several seemingly unrelated threads of research on first-order methods~\cite{BausBT16,BeckS15,LacoJ15,LuFN18,NecoNG18,PenaR16,Tebo18} and on condition measures for convex optimization~\cite{EpelF00,EpelF02,FreuV99b,Freu04,Lewi99,OrdoF03,Pena00,Rene95a,Rene95}.  Our  construction of $L_{f,X,D}$ and $\mu_{f,X,D}$ is inspired by and closely related to the work of Lu, Freund, and Nesterov~\cite{LuFN18} and of Bauschke, Bolte, and Teboulle~\cite{BausBT16,Tebo18}.  Lu et al.~\cite{LuFN18}  extend the concepts of smoothness and strong convexity constants by considering them {\em relative} to a {\em reference} function $h$, see~\cite[Definition 1.1 and 1.2]{LuFN18}.  Our construction is identical to theirs in the special case when the distance function is the Bregman distance function $D_h$ associated to  a reference function $h$ and the function $f$ is strictly convex.  Bauschke, Bolte, and Teboulle~\cite{BausBT16} define a concept of {\em Lipschitz-like} condition that is equivalent to smoothness relative to a reference function.     As we detail in Section~\ref{sec.algos}, our relative constants $L_{f,X,D}$ and $\mu_{f,X,D}$    are also identical to the {\em curvature constant, away curvature constant} and {\em geometric strong convexity constant} proposed by Jaggi~\cite{Jagg13} and by Lacoste-Julien and Jaggi in~\cite[Appendix C]{LacoJ15}    for properly chosen distance-like functions $D$.  Our constructions of  relative functional growth and relative quasi strong convexity are natural extensions of analogous concepts proposed by Necoara, Nesterov, and Glineur~\cite{NecoNG18} to unveil relaxations of strong convexity that ensure the linear convergence of first-order methods.  Our   relative functional growth concept is in the same spirit as that of the quadratic functional growth approach used by Beck and Shtern~\cite{BeckS15} to established the linear convergence of a conditional gradient algorithm with away steps for non-strongly convex functions.  

In contrast to the approaches in~\cite{BeckS15,LacoJ15,LuFN18,NecoNG18,PenaR16}, our construction of the relative condition constants applies to any pair  $(X,D)$ of reference set and distance function.  
Our main results (Section~\ref{sec.prop.rel} and Section~\ref{sec.prop.qrel}) reveal some interesting insights when $D$ is bounded by a squared norm.  We establish a close  connection between our relative conditioning approach and the conditioning of linear conic systems pioneered by Renegar~\cite{Rene95a,Rene95} and further developed by a number of authors~\cite{CheuC01,EpelF00,EpelF02,FreuV99b,Freu04,Lewi99,OrdoF03,Pena00,PenaVZ18}.  We especially draw on ideas developed in the recent paper~\cite{PenaVZ18}.  We note that consistent with our construction of the relative constants $L_{f,X,D}, \; \mu_{f,X,D}, \; \mu_{f,X,D}^\star, \; \mu^\sharp_{f,X,D}$,  all of our results concerning them scale appropriately, that is, they scale by $\lambda$ whenever the objective function $f$ is replaced by $\tilde f = \lambda f$ for some constant $\lambda > 0$.  In particular,  the relative condition number $L_{f,X,D}/\mu_{f,X,D}$ and all of our bounds on it are invariant under positive scaling of $f$.

The main sections of the paper are organized as follows. Section \ref{sec.relcond} presents our central construction, namely relative smoothness and relative strong convexity.  This section also introduces relative quasi strong convexity and relative functional growth, both of which are variants of relative strong convexity.  
Section~\ref{sec.prop.rel} and Section~\ref{sec.prop.qrel} present the main technical results of the paper. Section~\ref{sec.prop.rel} develops several properties of the constants $L_{f,X,D}$ and $\mu_{f,X,D}$. More precisely, Proposition~\ref{prop.smooth} gives an upper bound on $L_{f,X,D}$ when $f$ is of the form $g\circ A$ for some $A\in \R^{m\times n}, \; g:\R^m\rightarrow \R\cup\{\infty\}$.  
   Proposition~\ref{prop.smooth}(a) shows that the bound is tight.
The more involved Theorem~\ref{thm.conic} and Theorem~\ref{thm.poly} give lower bounds on  $\mu_{f,X,D}$ when $f$ is of the form $g\circ A$ and $X$ is a convex cone or a polyhedron.  These bounds readily imply that for $f = g\circ A$ the relative condition number $L_{f,X,D}/\mu_{f,X,D}$ can be bounded in terms of the product of the classical condition number $L_g/\mu_g$ and a condition number of the pair $(A,X)$.  See equation~\eqref{eq.rel.cond.conic} and equation~\eqref{eq.rel.cond.poly}.   Corollary~\ref{corol.conic} and Corollary~\ref{corol.main} show that the bounds in Theorem~\ref{thm.conic} and Theorem~\ref{thm.poly} are tight.
Section~\ref{sec.prop.qrel} develops properties analogous to those in Section~\ref{sec.prop.rel} but for the constants $\mu^\star_{f,X,D}$ and $\mu^\sharp_{f,X,D}$.    Section~\ref{sec.algos} details linear convergence results for the mirror descent algorithm, Frank-Wolfe algorithm, and Frank-Wolfe  with away steps algorithm for problem~\eqref{eq.Prob}.  In all cases the linear convergence properties are stated in terms of the relative constants $L_{f,X,D}$ and $\mu^\star_{f,X,D}, \mu^\sharp_{f,X,D}$ for suitable choices of distance-like function $D$.  
The main results in Section~\ref{sec.algos} can be summarized as follows.
Consider the mirror descent algorithm for problem~\eqref{eq.Prob} with a Bregman distance $D_h$ associated to a reference function $h:X\rightarrow \R$.  Proposition~\ref{prop.lin.mirror} shows the following linear convergence result: if $L_{f,X,D_h} < \infty$ and $\mu_{f,X,D_h}^\star > 0$ then the mirror descent iterates satisfy
\[
f(x_k) - f^\star \le L_{f,X,D_h} \left( 1 - \frac{\mu_{f,X,D_h}^\star}{L_{f,X,D_h}} \right)^k D_h(x^\star,x_0)
\]
for $x^\star \in \argmin_{x\in X} f(x)$.  Proposition~\ref{prop.lin.mirror.2} gives a linear convergence result of similar flavor when $\mu_{f,X,D_h}^\sharp > 0$.  The rates of convergence in both Proposition~\ref{prop.lin.mirror} and
Proposition~\ref{prop.lin.mirror.2} are at least as sharp, and possibly much sharper, than those in~\cite{LuFN18,Tebo18} and apply to a broader class of functions.  In particular, as Example~\ref{ex.counter.2} in Section~\ref{sec.prop.qrel} shows, there are instances  where $\mu_{f,X,D}^\sharp > \mu_{f,X,D} = 0$  occurs.  In such instances Proposition~\ref{prop.lin.mirror.2} yields the linear convergence of mirror descent whereas the linear convergence results in~\cite{LuFN18,Tebo18} do not apply.

Proposition~\ref{prop.lin.FW} gives a strikingly similar linear convergence 
result for the Frank-Wolfe algorithm: suppose $X$ is a compact convex set endowed with a linear oracle and $L_{f,X,\mathfrak{R}} < \infty$ and $\mu_{f,X,\mathfrak{R}}^\star > 0$  for the {\em radial distance function} $\mathfrak{R}:X\times X\rightarrow\R_+$ defined via~\eqref{eq.radial}. Proposition~\ref{prop.lin.FW} shows that the Frank-Wolfe iterates  
satisfy
\[
f(x_k) - f^\star \le \left( 1 - \frac{\mu_{f,X,\mathfrak{R}}^\star}{L_{f,X,\mathfrak{R}}} \right)^k(f(x_0) - f^\star).
\]
This rate of convergence subsumes and is sharper than the previously known linear convergence results for the Frank-Wolfe algorithm in~\cite{GuelM86,BeckT04}.

Proposition~\ref{prop.lin.fwa} gives a result of similar flavor for the Frank-Wolfe with away steps algorithm: suppose $X$ is a polytope endowed with a vertex linear oracle, and $L_{f,X,\mathfrak{D}} < \infty$ and $\mu_{f,X,\mathfrak{G}}^\star > 0$  for the distance functions $\mathfrak{D}:X\times X\rightarrow\R_+$ and $\mathfrak{G}:X\times X\rightarrow\R_+$ defined via~\eqref{eq.diametral} and~\eqref{eq.gradiental}.  Proposition~\ref{prop.lin.fwa} shows that if the Frank-Wolfe  with away steps algorithm starts from a vertex in $X$ then the subsequent iterates satisfy
\[
f(x_k) - f^\star \le \left( 1 - \min\left\{\frac{1}{2},\frac{\mu_{f,X,\mathfrak{G}}^\star}{4L_{f,X,\mathfrak{D}}}\right\} \right)^{k/2}(f(x_0) - f^\star).
\]
This rate of convergence is at least as sharp, and possible much sharper, than the rates previously shown in~\cite{BeckS15,LacoJ15,PenaR16}.

Throughout the paper we define a number of new objects that are necessary for our main developments.  To help the reader recall the definition and notation associated to these new objects, Table~\ref{table} displays the section and equation where each object is defined.

\begin{table}
\begin{center}
\begin{tabular}{c|c|c}
\hline
Symbol & Section & Equation \\
\hline
$Z_{f,X}(y)$ &\ref{sec.relcond}& \eqref{eq.Z.f.X}\\
$L_{f,X,D}$ and $\mu_{f,X,D}$ & \ref{subsec.def} &\eqref{eq.rel.constants}  \\
$L_{f}$ and $\mu_{f}$ & \ref{subsec.def} &\eqref{eq.classic.constants} \\
$\mu_{f,X,D}^\star$ and $\mu_{f,X,D}^\sharp$ &\ref{subsec.quasi} &\eqref{eq.qrel.const}  \\
$Z_{A,X}(y)$&\ref{sec.prop.rel}& \eqref{eq.Z.A.X} \\
$A|C$ and $(A|C)^{-1}$ &\ref{sec.prop.rel}&\eqref{eq.A.C} and \eqref{eq.A.C.inv}\\
$\|A|C\|$ and $\|(A|C)^{-1}\|$&\ref{sec.prop.rel}&\eqref{eq.norm.AC}\\
$\T(A|X)$&\ref{subsec.poly}& \eqref{eq.T.A.X}\\
$\Phi(A)$ and $\diam(A)$ &\ref{subsec.poly}&\eqref{eq.facial.dist.def} and \eqref{eq.diam.def}\\
$\T(A|X,S)$&\ref{subsec.sharp}& \eqref{eq.T.A.X.S}\\
\hline
\end{tabular}
\end{center}
\caption{Index of symbols introduced in the paper}
\label{table}
\end{table}

%
%
\section{Conditioning relative to a reference set and distance function pair}
\label{sec.relcond}

This section presents the central ideas of this paper.  We introduce the concepts of relative smoothness and relative strong convexity of a function relative to a reference set and distance function pair.  We also introduce some variants of relative strong convexity that are natural extensions of the approach developed by Necoara, Nesterov and Glineur~\cite{NecoNG18}. 

Throughout the entire paper we will make the following blanket assumption about the triple $(f,X,D)$.

\begin{assumption}\label{assump.blanket}
{\em The function $f:\R^n \rightarrow \R \cup\{\infty\}$ is convex and differentiable.  The set  $X \subseteq \dom(f)$ is convex. The function $D:X\times X \rightarrow \R_+$ is a reference {\em distance-like} function, that is, $D(y,x)\ge 0$ for all $x,y\in X$  and $D(x,x) = 0$ for all $x\in X$.}
\end{assumption}

Throughout our developments we will consider the following classes of reference distance-like functions:
\begin{itemize}
\item The {\em Bregman distance}  $D_h:X\times X\rightarrow \R_+$ associated to a  reference convex differentiable function $h:X\rightarrow \R$, that is, 
\[
D_h(y,x) :=
 h(y) - h(x) - \ip{\nabla h(x)}{y-x}.
\]
\item The square of a (non-necessarily Euclidean) norm $\|\cdot\|$ in $\R^n$, that is,
\[
D(y,x):=\frac{1}{2}\|y-x\|^2.
\]
{  
\item The square $\mathfrak{R}:=\frac{\mathfrak{r}^2}{2} $ of the {\em radial} distance function $\mathfrak{r}:X\times X\rightarrow \R_+$ defined as follows
\[
\mathfrak{r}(y,x):=\inf\{\rho > 0:  y-x=\rho\cdot(u-x) \text{ for some } u\in X\}.
\]
Notice that the function $v\mapsto \mathfrak{r}(x+v,x)$ coincides with the gauge function of the set $X-x$ on $X-x$. Figure~\ref{fig.radial} illustrates the level sets defined by $\mathfrak{r}(\cdot,x)$ for $X=\{x\in\R^2: \|x\|_2 \le 1\}$. 

\item The square $\mathfrak{D}:=\frac{\mathfrak{d}^2}{2}$ of the {\em diametral} distance function $\mathfrak{d}:X\times X\rightarrow \R_+$ defined as follows
\[
\mathfrak{d}(y,x):=\inf\{\delta > 0:  y-x=\delta\cdot(u-v) \text{ for some } u,v\in X\}.
\]
Figure~\ref{fig.diametral} illustrates the level sets defined by the diametral distance $\mathfrak{d}(\cdot,x)$ for $X=\{x\in\R^2: \|x\|_2 \le 1\}$. 
}
\end{itemize}

\begin{figure}
\begin{center}\includegraphics[width = 9cm]{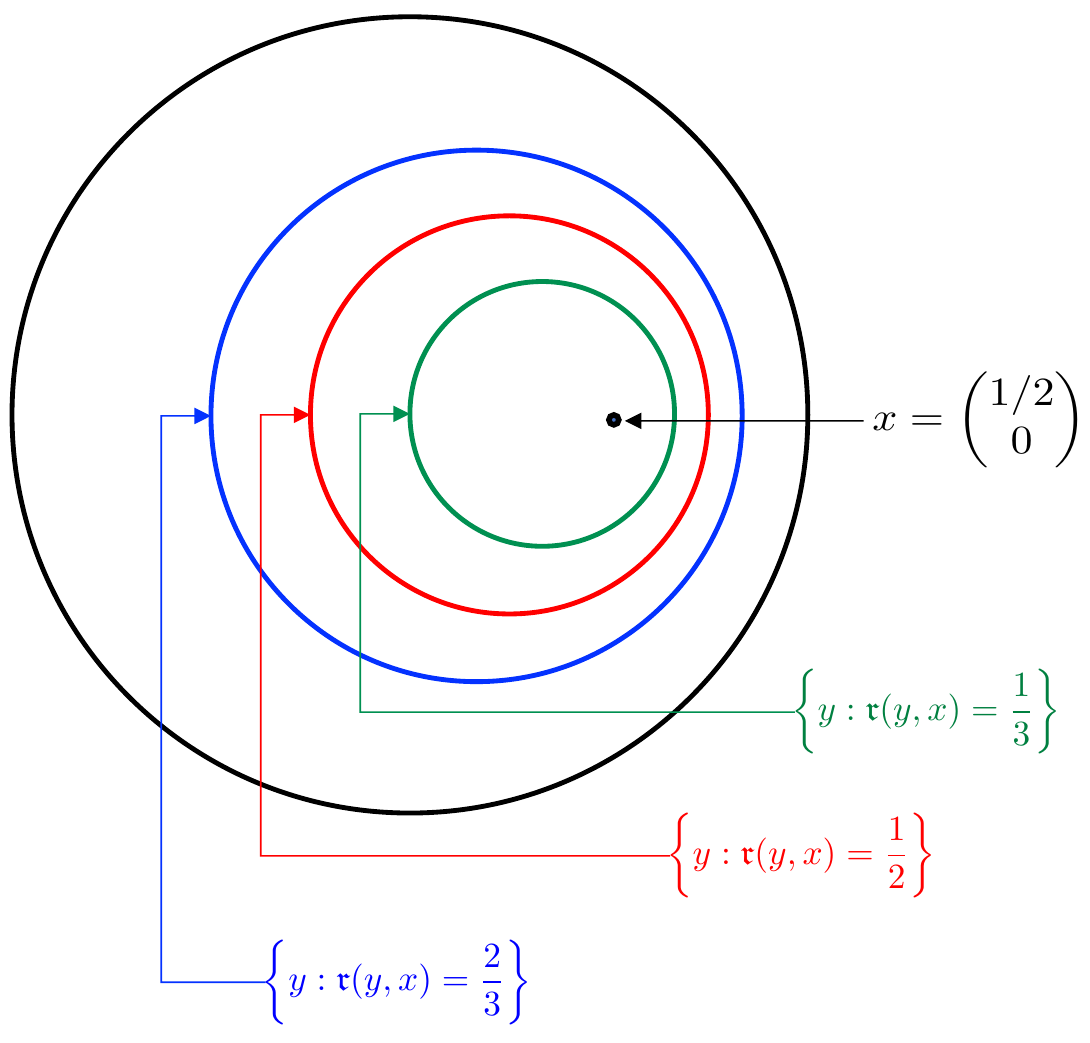}\end{center}
\caption{Level sets of $\mathfrak{r}(\cdot,x)$ in $X = \{x\in \R^2: \|x\|_2 \le 1\}$.}
\label{fig.radial}
\end{figure}

\begin{figure}
\begin{center}\includegraphics[width = 9cm]{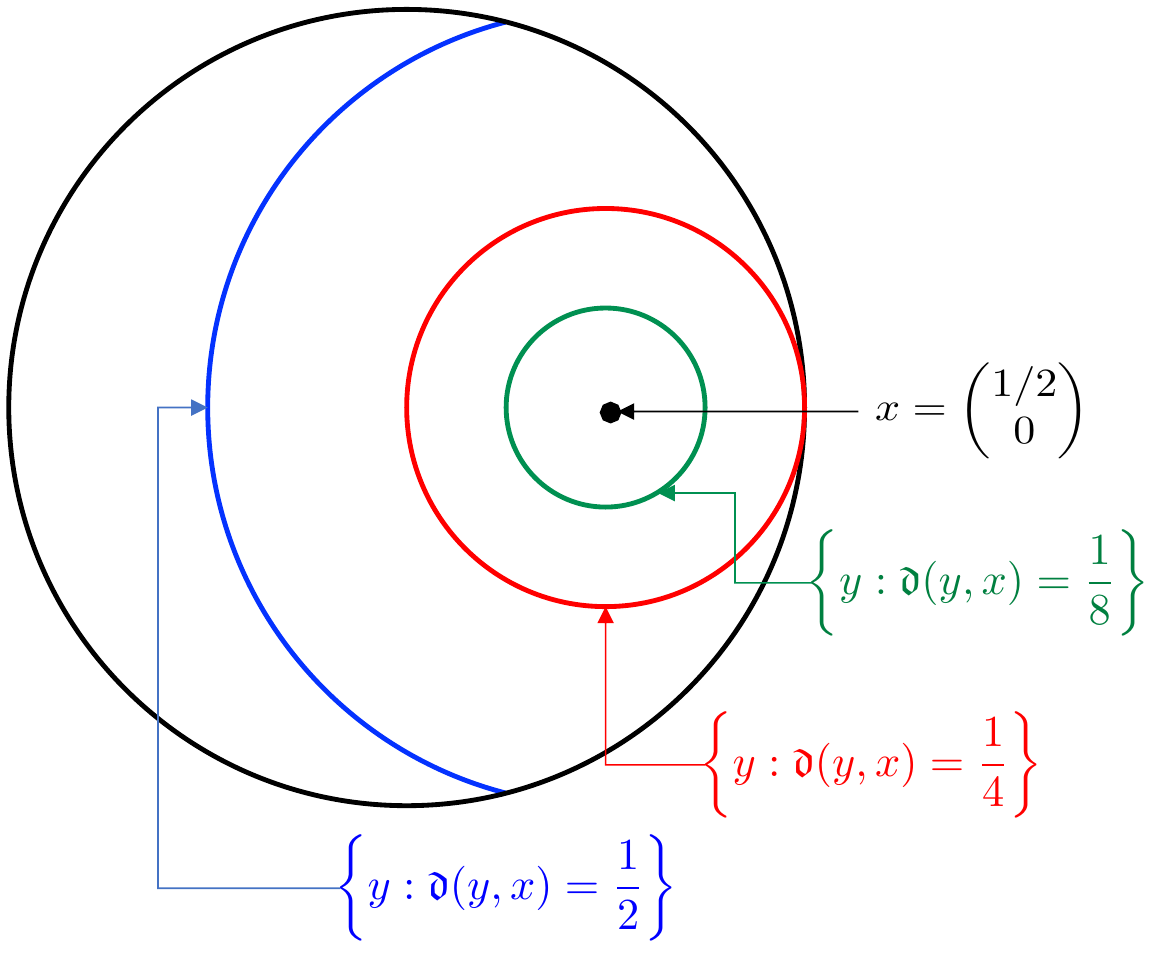}\end{center}
\caption{Level sets of $\mathfrak{d}(\cdot,x)$ in $X=\{x\in \R^2: \|x\|_2\le 1\}.$}
\label{fig.diametral}
\end{figure}


Our main construction is based on bounding the behavior of the  Bregman distance associated to $f$ in terms of the reference distance function $D$.  
The following set-valued mapping $Z_{f,X}:X \rightrightarrows X$ provides a key building block for our construction.  For $y\in X$ let $Z_{f,X}(y) \subseteq X$ denote the set
\begin{equation}\label{eq.Z.f.X}
Z_{f,X}(y):=\{x\in X: f(x) = f(y) \; \text{ and } \; \ip{\nabla f(x)-\nabla f(y)}{x-y} = 0\}.
\end{equation}
{  
It is easy to see that $Z_{f,X}(y)$ can also be written as
\[
Z_{f,X}(y) = \{x\in X: f(x+\lambda(y-x)) = f(y) \text{ for all } \lambda \in [0,1]\}. 
\]
Observe that if $f$ is strictly convex then $Z_{f,X}(y)=\{y\}$ for all $y\in X$.   The set $Z_{f,X}(y)$ captures the largest convex subset of  $\{x\in X: f(x) = f(y)\}$ that includes $y$  and where $f$ fails to be strictly convex.  
In particular, when $f$ is of the form $f= g\circ A$ for $A\in \R^{m \times n}$ and $g:\R^m \rightarrow \R\cup\{\infty\}$ strictly convex, it is easy to see that $Z_{f,X}(y) = \{x\in X: Ax = Ay\}$.  We will further discuss functions of this form in Section~\ref{sec.prop.rel} and Section~\ref{sec.prop.qrel}.  To illustrate the set-valued mapping $Z_{f,X}$ in a different example, consider the function $f:\R^n\rightarrow \R$ defined as
\[
f(x):=\min_{y\in \B^n} \|x-y\|_2^2,
\]
where $\B^n = \{y\in \R^n: \|y\|_2\le 1\}.$  In this case
\[
Z_{f,X}(y) = \left\{\begin{array}{ll}\{y\} & \text{if } y \not \in \B^n\\
\B^n & \text{if } y \in \B^n.
 \end{array}\right.
\]

}


%
%
\subsection{Relative smoothness and relative strong convexity}
\label{subsec.def}

To motivate our main construction we first recall the classical notion of smoothness and strong convexity constants.  We recall these classical concepts in a format that we subsequently use for our main construction. Recall that for a convex differentiable function $f:\R^n \rightarrow \R \cup \{\infty\}$ and $x,y\in \dom(f)$ the Bregman distance $D_f(y,x)$ is
\[
D_f(y,x) = f(y) - f(x) - \ip{\nabla f(x)}{y-x}.
\]
\begin{definition}\label{def.regular}{\em Suppose 
$f:\R^n\rightarrow \R\cup\{\infty\}$ is convex and differentiable and $D(y,x) = \frac{1}{2}\|y-x\|^2$ for
some norm $\|\cdot\|$ in
$\R^n$.
\begin{itemize}
\item[(a)] The function $f$ is smooth for the norm $\|\cdot\|$ if there exists a constant $L>0$  such that 
\begin{equation}\label{eq.smooth}
D_f(y,x) \le LD(y,x) \; \text{ for all } \; x,y\in \dom(f).
\end{equation}
\item[(b)] The function $f$ is strongly convex for the norm $\|\cdot\|$ if there exists a constant $\mu>0$  such that 
\begin{equation}\label{eq.strong.conv}
D_f(y,x) \ge \mu D(y,x)  \; \text{ for all } \; x,y\in \dom(f).
\end{equation}
\end{itemize}
}
\end{definition}

Next, we present our main construction.  In Definition~\ref{def.relative} and throughout the paper we will use the following notational convention.  For a nonempty $S\subseteq X$ and $x\in X$ let $D_f(S,x)$ and $D(S,x)$ denote $\inf_{y\in S} D_f(y,x)$ and  $\inf_{y\in S} D(y,x)$ respectively.


\begin{definition}\label{def.relative}{\em Let $(f,X,D)$ satisfy Assumption~\ref{assump.blanket}.
\begin{itemize}
\item[(a)] We say that $f$ is {\em smooth relative} to $(X,D)$ if there exists a constant $L >0$ such that
\begin{equation}\label{eq.smooth.rel}
D_f(y,x) \le LD(y,x) \; \text{ for all } \; x,y\in X.
\end{equation}
\item[(b)] We say that $f$ is {\em strongly convex relative} to $(X,D)$ if there exists a constant $\mu >0$ such that
\begin{equation}\label{eq.strong.conv.rel}
D_f(Z_{f,X}(y),x) \ge \mu D(Z_{f,X}(y),x) \; \text{ for all } \; x,y\in X.
\end{equation}
\end{itemize}
}
\end{definition}

When $D=D_h$ for some convex differentiable function $h:X\rightarrow\R$, the above relative smoothness concept is identical to  the smoothness of $f$ relative to $h$ on $X$ as defined in~\cite{LuFN18}.  The latter in turn is equivalent to the {\em Lipschitz-like condition} defined in~\cite{BausBT16}.  Furthermore, when $D=D_h$ and $f$ is strictly convex, the above relative strong convexity concept is identical to the strong convexity of $f$ relative to $h$ on $X$ as defined in~\cite{LuFN18}.     We note that as in~\cite{LuFN18}, the above definitions~\eqref{eq.smooth.rel} and~\eqref{eq.strong.conv.rel} are not symmetric in $x$ and $y$ since they depend on $D_h$ and $D$ which are not 
necessarily symmetric.  Observe that the term $Z_{f,X}(y)$ instead of $y$ in~\eqref{eq.strong.conv.rel} makes this definition of relative strong convexity less stringent than the classical one~\eqref{eq.strong.conv} or the one in~\cite{LuFN18}.  This is a key feature of our construction. 

We will use the following notation throughout the rest of the paper.  Suppose $(f,X,D)$ satisfies Assumption~\ref{assump.blanket}. Let $L_{f,X,D}$ and $\mu_{f,X,D}$ be the following relative smoothness and strong convexity constants
\begin{equation}\label{eq.rel.constants}
L_{f,X,D}:=\inf\{L>0: \eqref{eq.smooth.rel} \text{ holds}\}, \; \mu_{f,X,D}:=\sup\{\mu\ge 0: \eqref{eq.strong.conv.rel} \text{ holds}\}.
\end{equation}
In addition, suppose $f:\R^n\rightarrow \R\cup\{\infty\}$ is convex and differentiable and $D(y,x) = \frac{1}{2}\|y-x\|^2$ for
some norm $\|\cdot\|$ in $\R^n$.  Let $L_f$ and $\mu_f$ be the following classical smoothness and strong convexity constants
\begin{equation}\label{eq.classic.constants}
L_{f}:=\inf\{L>0: \eqref{eq.smooth} \text{ holds}\}, \; \mu_{f}:=\sup\{\mu\ge 0: \eqref{eq.strong.conv} \text{ holds}\}.
\end{equation}
The following example illustrates the values of the  relative smoothness and strong convexity constants $L_{f,X,D}$ and $\mu_{f,X,D}$ of a convex quadratic function relative to $(X,D)$ for some canonical choices of $f,X,$ and $D$.  
   Example~\ref{example.sing.vals} highlights that the relative constants $L_{f,X,D}$ and $\mu_{f,X,D}$ depend on the combination of the constraint set $X$ and the function $f$.  In particular, Example~\ref{example.sing.vals} shows that the relative condition number $L_{f,X,D}/\mu_{f,X,D}$ can be vastly different (both smaller or larger) than the usual condition number $L_f/\mu_f$ depending on how the shape of $X$ fits $f$.
Example~\ref{example.sing.vals} also lays the ground for the main properties that we develop in Section~\ref{sec.prop.rel}.

\begin{example}\label{example.sing.vals}{\em Let $A\in \R^{m\times n}, b\in \R^m$ with $A\ne 0$ and $\R^n$ and $\R^n$ be endowed with the Euclidean norm.
  Let $f(x) =  \frac{1}{2} \|Ax-b\|_2^2$   and $D(y,x)=\frac{1}{2}\|y-x\|_2^2.$  Then $f$ has the following smoothness and strong convexity constants $L_{f,X,D}$ and $\mu_{f,X,D}$ relative to $(X,D)$ for   some particular choices of $X$. 
\begin{itemize}
\item[(a)] For $X = \R^n$ we have $L_{f,X,D} = \sigma_{\max}(A\transp A) = \sigma_{\max}(A)^2$ and $\mu_{f,X,D} = \sigma_{\min}^+(A\transp A) = \sigma_{\min}^+(A)^2 > 0$, where $\sigma_{\min}^+(\cdot)$ denotes the smallest {\em positive} singular value. Observe that in this case $L_f = L_{f,X,D}$ but $\mu_f = \mu_{f,X,D}$ only when $A$ is full column rank.

\item[(b)] Suppose $X \subseteq \R^n$ is a linear subspace such that the mapping $A|X:X\rightarrow \R^m$ defined via $x\in X \mapsto Ax \in \R^m$ is nonzero.  Then $L_{f,X,D} =  \sigma_{\max}(A|X)^2$ and $\mu_{f,X,D} = \sigma_{\min}^+(A|X)^2$.  Observe that in this case $L_{f,X,D}\le L_f$ and $L_{f,X,D}$ can be quite a bit smaller.  Likewise, $\mu_{f,X,D} \ge \mu_f$ and 
$\mu_{f,X,D}$ can be quite a bit larger.

{  
For instance, suppose $A = \text{diag}(I_{n-2},M,\epsilon)\in\R^{n\times n}$ for some positive $M,\epsilon$ with $0<\epsilon \ll 1 \ll M$.  If $X = \R^{n-2} \times \{0_2\}\subseteq \R^n$ then $$\mu_f = \epsilon^2 \ll 1 = \mu_{f,X,D} = L_{f,X,D} \ll M^2 = L_f.$$
In this case we have $L_{f,X,D}/\mu_{f,X,D} \ll L_f/\mu_f$.
\item[(c)] Suppose $X = \R^n_+$. In this case $L_{f,X,D} = \|A\|^2 = \sigma_{\max}(A\transp A) = L_f$.  On the other hand, if $A(\R^n_+) =\R^m$ then
$\mu_{f,X,D}$ is the following kind of squared {\em signed} smallest singular 
value of $A$ 
$$\mu_{f,X,D}  = \left(\max\{r: r\B^m \subseteq A(\B^n\cap \R^n_+)\}\right)^2,$$    
where $\B^m$ and $\B^n$ denote the unit balls in $\R^m$ and   $\R^n$ respectively. In other words, $\mu_{f,X,D}$ is the square of the radius of the largest ball centered at zero and contained in $A(\B^n\cap \R^n_+)$.  Observe that if $X=\R^n_+$ and $A(\R^n_+) = \R^m$ then $0< \mu_{f,X,D} \le \sigma_{\min}(A)^2$ and $\mu_{f,X,D}$ can be quite a bit smaller.  For instance, if $A = \matr{1 & -1 & 0 \\-\epsilon & -\epsilon & 1}$ for $0 < \epsilon \ll 1$ then 
\[
\mu_{f,X,D} = 2 \epsilon^2 \ll 1 + 2\epsilon^2 = \sigma_{\min}(A)^2 = \mu_f.
\]
In this case we have $L_{f,X,D}/\mu_{f,X,D} \gg L_f/\mu_f$.
}
\end{itemize}
}
\end{example}

The  statements (a), (b), and (c) in Example~\ref{example.sing.vals} can be verified directly but they also follow from the more general Proposition~\ref{prop.smooth}, Corollary~\ref{corol.conic},    and Corollary~\ref{corol.main} in Section~\ref{sec.prop.rel} below.  

\subsection{Relative quasi strong convexity and relative functional growth}
\label{subsec.quasi}

Following~\cite{NecoNG18}, we next consider two variants of relative strong convexity that are natural extensions of the 
{\em quasi-strong convexity} and {\em quadratic functional growth} concepts defined in~\cite{NecoNG18}.  For that purpose, we will rely on the following strengthening of Assumption~\ref{assump.blanket}.

\begin{assumption}\label{assump.min}
{\em Suppose $(f,X,D)$ satisfy Assumption~\ref{assump.blanket}, $f^\star := \min_{x\in X}f(x)$ is finite, $X^\star:=\{x\in X: f(x) = f^\star\} \ne \emptyset$, and the map $x \mapsto \bar x:=\argmin_{y\in X^\star} D(y,x)$ is well defined for all $x\in X$.
}
\end{assumption}


\begin{definition}\label{def.qrelative}{\em Suppose $(f,X,D)$ satisfies Assumption~\ref{assump.min}.  
\begin{itemize}
\item[(a)] We say that $f$ is {\em quasi-strongly-convex relative} to $(X,D)$ if there exists a constant $\mu > 0$ such that
\begin{equation}\label{eq.quasi.strong.convex.rel}
D_f(\bar x,x) \ge \mu D(\bar x,x) \; \text{ for all } \; x\in X.
\end{equation}
\item[(b)] We say that $f$ has {\em $D$-relative functional growth} on $X$ if there exists a constant $\mu > 0$ such that
\begin{equation}\label{eq.func.growth}
f(x) - f^\star \ge \mu D(\bar x,x) \; \text{ for all } \; x\in X.
\end{equation}
\end{itemize}
}
\end{definition}
Throughout the sequel we will use the following notation analogous to~\eqref{eq.rel.constants}.  Suppose $(f,X,D)$ satisfies Assumption~\ref{assump.min}.  Let $\mu_{f,X,D}^\star$ and $\mu_{f,X,D}^\sharp$ 
be as follows
\begin{equation}\label{eq.qrel.const}
\mu_{f,X,D}^\star:=\sup\{\mu\ge 0: \eqref{eq.quasi.strong.convex.rel} \text{ holds}\}, \; \mu_{f,X,D}^\sharp:=\sup\{\mu\ge 0: \eqref{eq.func.growth} \text{ holds}\}.
\end{equation}

The next proposition shows that, as one may intuitively expect,  relative quasi-strong  convexity  is  a relaxation of relative strong  convexity. In other words, $\mu_{f,X,D} \le \mu_{f,X,D}^\star$ whenever $(f,X,D)$ satisfies Assumption~\ref{assump.min}.

\begin{proposition}\label{prop.qrel} Suppose $(f,X,D)$ satisfy Assumption~\ref{assump.min}.   If $\mu>0$ is such that $(f,X,D,\mu)$ satisfies~\eqref{eq.strong.conv.rel} then $(f,X,D,\mu)$ satisfies~\eqref{eq.quasi.strong.convex.rel}.

\end{proposition}
\begin{proof}
The construction of $Z_{f,X}(y)$ implies that $Z_{f,X}(y) = X^\star$ for all $y\in X^\star$.  Therefore, if $(f,X,D,\mu)$ satisfies~\eqref{eq.strong.conv.rel} then by taking $y=\bar x$ it follows that
\[
 D_f(\bar x,x) \ge D_f(Z_{f,X}(\bar x),x) \ge \mu  D(Z_{f,X}(\bar x),x) = \mu D(\bar x,x) \; \text{ for all } \; x\in X.
\]
\end{proof}

The following simple example shows that, perhaps contrary to what one might intuitively expect,   relative functional growth  is not necessarily a relaxation of  strong relative convexity  unless some additional assumptions are made about $f, X,$ or $D$.   

\begin{example}\label{ex.counter.1}{\em
Let $a > 0$ and $f:\R \rightarrow \R$ be the function $f(x)= e^{ax}$.  For $X:=\R_+$ we have $X^\star = \{0\}$.  Thus for  $D := D_f$ and $\mu = 1$ the tuple $(f,X,D,\mu)$ satisfies \eqref{eq.strong.conv.rel}.  However, observe that  for all $\hat \mu > 0$ and $x \ge 1/(\hat\mu a)$
\[
f(x) - f^\star = e^{ax} - 1 < \hat\mu(1+axe^{ax}) = \hat\mu(f^\star - f'(x)(0-x)) = \hat\mu D(X^\star,x).  
\]
In particular, $(f,X,D,\hat\mu)$ does not satisfy \eqref{eq.func.growth} for any $\hat \mu > 0$.
}
\end{example}

It can be shown that under additional assumptions on $f, X,$ or $D$ the relative functional growth condition is a relaxation of the relative strong convexity condition.  In particular,  relative functional growth is a relaxation of 
relative strong convexity  when $D$ is a squared norm as we discuss in Section~\ref{sec.prop.qrel} below.

%
%
\section{Properties of $L_{f,X,D}$ and $\mu_{f,X,D}$ when 
$f$ is of the form $g\circ A$}
\label{sec.prop.rel}
This section develops some properties of the relative constants $L_{f,X,D}$ and $\mu_{f,X,D}$ when $f$ is of the form $f:=g\circ A$ for $A\in\R^{m\times n}$ and $g:\R^m\rightarrow \R\cup \{\infty\},$ and $D$  is bounded in terms of some norm in $\R^n$.  The main results of this section are  Theorem~\ref{thm.conic} and Theorem~\ref{thm.poly}.  These results provide lower bounds on $\mu_{f,X,D}$ in terms of $\mu_g$ and the norms of some canonical set-valued mappings that depend on $A$ and $X$.  In a similar vein, Proposition~\ref{prop.smooth} gives an upper bound  on $L_{f,X,D}$ in terms of $L_g$ and the norm of a canonical mapping associated to $A$ and $X$.

\medskip

We will rely on the objects $Z_{A,X}(\cdot)$ and $A|C, (A|C)^{-1}$  defined next.  For $A\in \R^{m\times n}, \; X\subseteq \R^n$ nonempty and $y\in X$ let 
\begin{equation}\label{eq.Z.A.X}
Z_{A,X}(y):=\{x\in X: Ax = Ay\}.
\end{equation}
The set-valued mapping $Z_{A,X}: X\rightrightarrows X$ can be seen as an extension of the set-valued mapping $Z_{f,X}: X\rightrightarrows X$ introduced in Section~\ref{subsec.def}.

For $A\in \R^{m\times n}$ and a convex cone $C \subseteq \R^n$ let $A|C:\R^n\rightrightarrows \R^m$ be the set-valued mapping defined via
\begin{equation}\label{eq.A.C}
x \mapsto (A|C)(x) := \left\{\begin{array}{ll}
\{Ax\} & \text{ if } \; x\in C \\
\emptyset & \text{ otherwise, }\end{array}\right.
\end{equation}
and let $(A|C)^{-1}:\R^m\rightrightarrows \R^n$ be its inverse, that is,
\begin{equation}\label{eq.A.C.inv}
v \mapsto (A|C)^{-1}(v):=\{x\in C: Ax = v\}.
\end{equation}
Suppose $\R^n$ and $\R^m$ are endowed with norms.  Define the norms of $A|C$ and of $(A|C)^{-1}$ as follows
\begin{equation}\label{eq.norm.AC}
\|A|C\| := \sup_{x\in C\atop \|x\|\le 1} \|Ax\|, \;\; 
\|(A|C)^{-1}\| := \sup_{v\in A(C) \atop \|v\|\le 1} \inf_{x\in C \atop Ax = v} \|x\|.
\end{equation}

Observe that if $A\in \R^{m\times n}$ and $X \subseteq \R^n$ is a  convex set that contains more than one point then
\begin{equation}\label{eq.norm.map}
\|A|\lspan(X-X)\| = \sup_{y,x\in X \atop x\ne y} \frac{\|Ay-Ax\|}{\|y - x\|},
\end{equation}
{  where $\lspan(X-X)$ denotes the linear subspace spanned by $X-X$, that is,
\[
\lspan(X-X) = \{\lambda(x - y): x,y\in X, \lambda\in \R\}.
\]
}
In particular, the following property of the relative smoothness constant readily follows.

\begin{proposition}\label{prop.smooth}  Let $A\in \R^{m\times n}$ and $X \subseteq \R^n$ be a convex set that contains more than one point.
\begin{description}
\item[(a)] If $\R^m$ is endowed with the Euclidean norm,    
$D(y,x)= \frac{1}{2}\|y-x\|^2$ for some norm in $\R^n$, and $f(x) = \frac{1}{2}\|Ax-b\|_2^2$ for some $b\in \R^m$ then
\[
L_{f,X,D} = \|A|\lspan(X-X)\|^2.
\]
\item[(b)] 
Suppose $\R^m,\R^n$ are endowed with norms and    $D(y,x)\ge \frac{1}{2}\|y-x\|^2$  for the norm in $\R^n$. 
If $f = g\circ A$ where $g:\R^m\rightarrow \R\cup\{\infty\}$ is $L_g$ smooth for the norm in $\R^m$ then
\[
L_{f,X,D} \le L_g \|A|\lspan(X-X)\|^2.
\]
\end{description}
\end{proposition}
\begin{proof}
\begin{description}
\item[(a)] This follows from~\eqref{eq.norm.map} and $D_{f}(y,x) =  \frac{1}{2}\|Ay-Ax\|^2_2$.
\item[(b)] This follows from~\eqref{eq.norm.map} and  $D_{f}(y,x) =  D_g(Ay,Ax) \le \frac{L_g}{2}\|Ay-Ax\|^2$.  The latter inequality follows from the $L_g$ smoothness of $g$. 
\end{description}
\end{proof}

We next discuss far more interesting results that either characterize or lower bound  the relative strong convexity constant $\mu_{f,X,D}$. 

\subsection{Lower bound on $\mu_{f,X,D}$ when $X$ is a convex cone and $A(X)$ is a linear subspace}
\label{subsec.conic}

In this subsection we will consider the special case when $X\subseteq\R^n$ is a convex cone and $A\in \R^{m\times n}$ is such that $A(X)$ is a linear subspace of $\R^m$.  The latter condition is equivalent to the following {\em Slater condition:} there exists $x\in \relint(X)$ such that $Ax=0$,   where $\relint(X)$ denotes the relative interior of $X$.   When this is the case, the norms $\|A|X\|$ and $\|(A|X)^{-1}\|$ have the following geometric interpretation. Let $\B^m$ and $\B^n$ denote the unit balls in $\R^m$ and $\R^n$ respectively.  It is easy to see that if $X$ is a convex cone and $A(X)$ is a linear subspace then  
\begin{equation}\label{eq.norm} 
\|A|X\| = \inf\{r : A(X\cap \B^n) \subseteq r \B^m \cap A(X)\}
\end{equation}
and
\begin{equation}\label{eq.norm.inv.map} 
\frac{1}{\|(A|X)^{-1}\|} = \sup\{r :  r \B^m \cap A(X)\subseteq A(X\cap \B^n) \}.
\end{equation}
In other words, $\|A|X\|$ is the radius of the {\em smallest} ball in $A(X)$ centered at the origin that {\em contains} $A(X\cap \B^n)$.  Similarly, $1/\|(A|X)^{-1}\|$ is the radius of the {\em largest} ball in $A(X)$ centered at the origin and that is {\em contained} in $A(X\cap \B^n)$.  {  
Example~\ref{ex.norms} illustrates this geometric interpretation of $\|A|X\|$ and $1/\|(A|X)^{-1}\|$ in a simple instance.

\begin{example}\label{ex.norms} 
{\em Let $A := \matr{1 & -1 & 0 \\-\epsilon & -\epsilon & 1}$ for $0 < \epsilon < 1$ and $X  = \R^3_+$. Let $\R^2$ be endowed with the Euclidean $\ell_2$ norm and let $\R^3$  be endowed with the $\ell_1$ norm.  In this case $A(X) = \R^2$ and $$A(X\cap \B^3) = \conv\left\{\matr{1\\-\epsilon},\matr{-1\\-\epsilon},\matr{0\\1} \right\}.$$ Therefore $\|A|X\| = \sqrt{1+\epsilon^2}$ and $1/\|(A|X)^{-1}\| = \epsilon$ as Figure~\ref{fig.norms} illustrates.

\begin{figure}
\begin{center}\includegraphics[width=10cm]{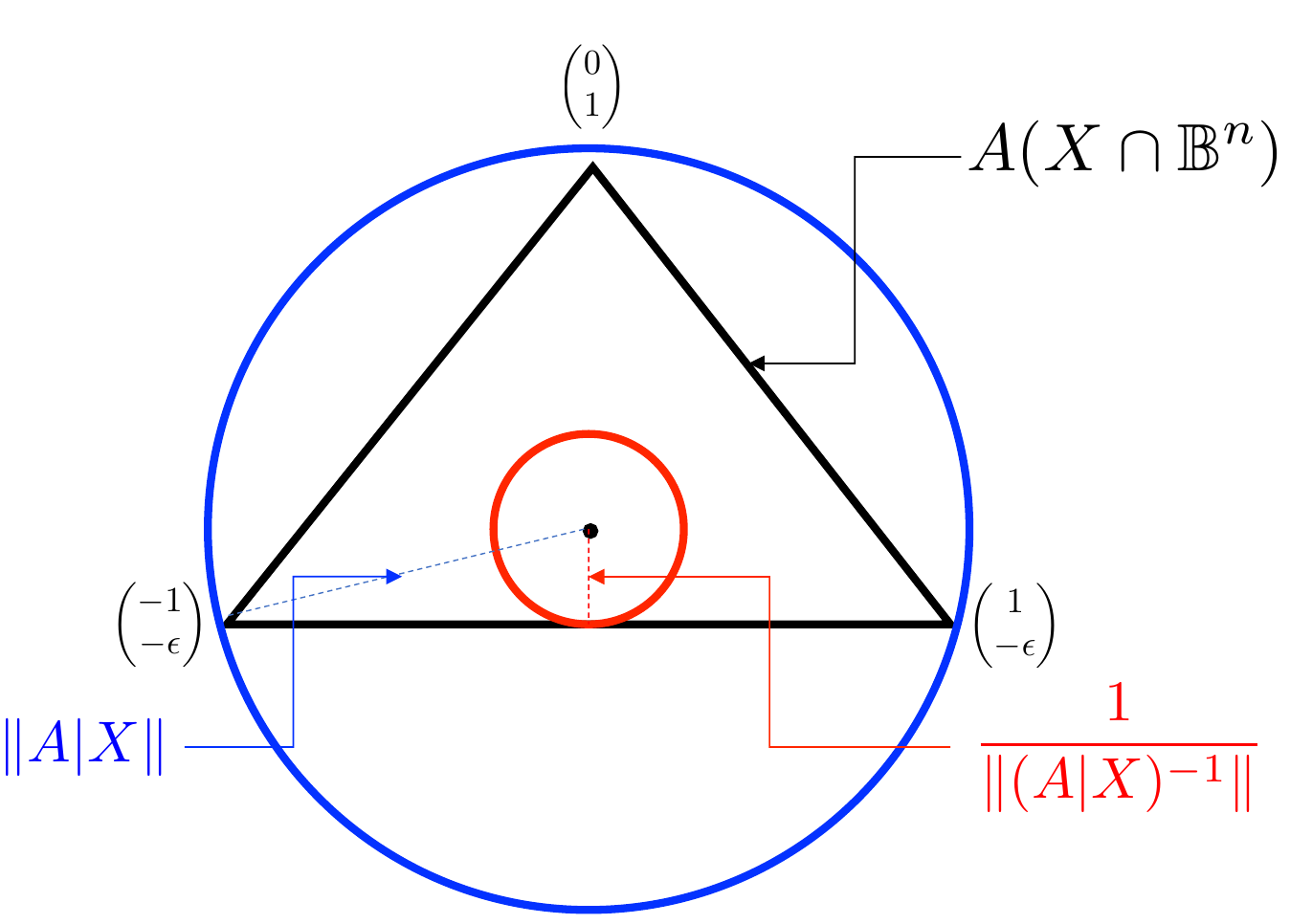}\end{center}
\caption{Illustration of $\|A|X\|$ and $1/\|(A|X)^{-1}\|$ for $A$ and $X$ as in Example~\ref{ex.norms}.}
\label{fig.norms}
\end{figure}

}
\end{example}
}

The above norms, especially $\|(A|X)^{-1}\|$ and other related quantities, have been extensively studied in the literature on condition measures for convex optimization~\cite{CheuC01,EpelF02,Freu04,Pena00,Rene95,Rene95a}. 
 They have been further extended to the broader variational analysis context~\cite{Lewi99,DontLR03}.  In particular, when $A(X) = \R^m$ the family of conic systems $Ax = b, x\in X$ is {\em well-posed.} That is, for all $b\in \R^m$ the conic system $Ax = b, x\in X$ is feasible and remains so for sufficiently small perturbations of $(A,b)$.  In this case it follows from~\cite{Rene95} that the quantity $1/\|(A|X)^{-1}\|$ is precisely the {\em distance to ill-posedness} introduced by Renegar~\cite{Rene95a,Rene95}, that is, the size of the smallest perturbation $\Delta A$ on $A$ so that the conic system $(A+\Delta A)x = b, x\in X$ is infeasible for some  $b\in\R^m$.  A similar identity holds for the {\em distance to non-surjectivity} of closed sublinear set-valued mappings~\cite{Lewi99}.  The latter in turn extends to a far more general identity for the radius of metric regularity~\cite{DontLR03}.

Observe that if $A\in \R^{m\times n}$ and  $X\subseteq \R^n$ is a linear subspace then $A(X)$ is automatically a linear subspace.  If in addition $\R^n$ and $\R^m$ are each endowed with  Euclidean norms, then~\eqref{eq.norm} and~\eqref{eq.norm.inv.map} yield
\[
\|A|X\| = \sigma_{\max}(A|X) \; \text{ and } \; \frac{1}{\|(A|X)^{-1}\|} = \sigma_{\min}^+(A|X).
\]

Corollary~\ref{corol.conic} and Theorem~\ref{thm.conic} below show that there is a tight connection between the relative strong convexity constant $\mu_{f,X,D}$ and the norm $\|(A|X)^{-1}\|$ when $f$ is of the form $g\circ A$.  Both of these results rely on the following proposition that characterizes a certain type of {\em Hoffman} constant~\cite{Hoff52}.  Proposition~\ref{prop.norm} is closely related to developments in~\cite{PenaVZ18,RamdP16}.   Proposition~\ref{prop.norm}  extends \cite[Theorem 2]{RamdP16} that only applies to the case $X = \R^n_+$.

\begin{proposition}\label{prop.norm} Suppose $\R^n$ and $\R^m$ are endowed with norms.  Let $A\in \R^{m\times n}$ and $X \subseteq \R^n$ be a convex cone such that $A(X)$ contains more than one point.  If $A(X)$ is a linear subspace then
\begin{equation}\label{eq.norm.inv}
\frac{1}{\|(A|X)^{-1}\|} = \inf_{x,y\in X \atop x \not\in Z_{A,X}(y)}\frac{\|Ay-Ax\|}{\|Z_{A,X}(y) - x\|}.
\end{equation}
\end{proposition}
\begin{proof}
Fix $y\in X$ and $x \in X\setminus Z_{A,X}(y)$.  Since $A(X)$ is a linear subspace, it follows  that $Ay-Ax \in A(X)$ and thus $Ay-Ax = Au$ for some $u\in X$ with $\|u\| \le \|(A|X)^{-1}\|\cdot \|Ay-Ax\|.$  Hence $x+u \in Z_{A,X}(y)$ and $\|Z_{A,X}(y)-x\| \le \|u\| \le \|(A|X)^{-1}\|\cdot \|Ay-Ax\|.$ Since this holds for 
arbitrary $y\in X$ and $x \in X\setminus Z_{A,X}(y)$ we conclude that
\[
\frac{1}{\|(A|X)^{-1}\|} \le \inf_{y,x\in X\atop x\not\in Z_{A,X}(y)} \frac{\|Ay-Ax\|}{\|Z_{A,X}(y)-x\|}.
\]
To prove the reverse inequality, let $v \in A(X)$    and $0<\epsilon< \|(A|X)^{-1}\|$ be such that $\|v\| =1$ and $\|y\| \ge \|(A|X)^{-1}\|-\epsilon$ for all $y\in X$ with $Ay = v$.  Pick $\hat y \in X$ with $A\hat y=v$. Then
$\|z\| \ge \|(A|X)^{-1}\|-\epsilon>0$ for all $z\in Z_{A,X}(\hat y)$.  Thus $\hat x := 0  \in X\setminus Z_{A,X}(\hat y)$ and
\[
\frac{1}{\|(A|X)^{-1}\|-\epsilon} \ge \frac{\|A\hat y-A\hat x\|}{\|Z_{A,X}(\hat y)-\hat x\|} \ge  \inf_{y,x\in X\atop x\not\in Z_{A,X}(y)} \frac{\|Ay-Ax\|}{\|Z_{A,X}(y)-x\|}.
\]
To finish let $\epsilon \rightarrow 0$.
\end{proof}

Proposition~\ref{prop.norm} readily yields the following result that generalizes Example~\ref{example.sing.vals}.

\begin{corollary}\label{corol.conic} Suppose $\R^m$ is endowed with the Euclidean norm $\|\cdot\|_2$, $\R^n$ is endowed with a norm $\|\cdot\|$, and $D(x,y) = \frac{1}{2}\|x-y\|^2$.  
If $f(x) = \frac{1}{2} \|Ax - b\|_2^2$ for some $A\in \R^{m\times n}$ and $b\in \R^m$, $X\subseteq \R^n$ is a convex cone, and $A(X)$ is a linear subspace that contains more than one point then
\[
\mu_{f,X,D} = \frac{1}{\|(A|X)^{-1}\|^2}.
\]
\end{corollary}
\begin{proof}
This follows from Proposition~\ref{prop.norm} and the observation that for this choice of $f$ and $X$ we have $Z_{f,X}(y) = Z_{A,X}(y)$ and $f(y) - f(x) - \ip{\nabla f(x)}{y-x} =\frac{1}{2} \|Ay-Ax\|_2^2.$
\end{proof}

The following result  extends Corollary~\ref{corol.conic} to a broader class of functions.  

\begin{theorem}\label{thm.conic} 
Suppose $\R^n$ and $\R^m$ are endowed with norms and {  $D(x,y)  \le   \frac{1}{2}\|x-y\|^2$} for the norm $\|\cdot\|$ in $\R^n$.  Let $A\in\R^{m\times n}, \; g:\R^m \rightarrow \R\cup \{\infty\}$ be a convex differentiable function, and $X\subseteq \R^n$ be a convex cone such that $A(X)$ is a linear subspace that contains more than one point.  If $g$ is $\mu_g$ strongly convex for the norm $\|\cdot\|$ in $\R^m$ 
then the function $f = g \circ A$ satisfies
\[
\mu_{f,X,D} \ge \frac{\mu_g}{\|(A|X)^{-1}\|^2}.
\]
\end{theorem}
\begin{proof} 
Observe that 
$
D_f(y,x) = g(Ay) - g(Ax) - \ip{g(Ax)}{A(y-x)}
$
for all $y,x\in X$ 
Since $g$ is $\mu_g$ strongly convex, it follows that $D_f(y,x) \ge \mu_g\|Ay-Ax\|^2/2$ for all $y,x\in X$ and $Z_{f,X}(y) = \{x\in X: Ax = Ay\} = Z_{A,X}(y)$ for all $y\in X$.  Therefore Proposition~\ref{prop.norm} implies that
\[
\mu_{f,X,D} \ge \inf_{y,x\in X\atop x\not\in Z_{A,X}(y)} \frac{D_f(y,x)}{\|Z_{A,X}(y) - x\|^2/2} \ge 
\inf_{y,x\in X\atop x\not\in Z_{A,X}(y)} \frac{\mu_g\|Ay-Ax\|^2}{\|Z_{A,X}(y) - x\|^2} =  \frac{\mu_g}{\|(A|X)^{-1}\|^2}.
\]
\end{proof}

If $f,X,D$ are as in Corollary~\ref{corol.conic} then by Proposition~\ref{prop.smooth}  the relative condition number $L_{f,X,D}/\mu_{f,X,D}$ is \[
\frac{L_{f,X,D}}{\mu_{f,X,D}} = 
\left(\|A|\lspan(X)\|\cdot \|(A|X)^{-1}\| \right)^2\] which has a striking resemblance to the classical condition number~\eqref{eq.classic.cn} of $f(x) = \frac{1}{2}\|Ax-b\|_2^2.$  More generally, if $f,X,D$ are as in Theorem~\ref{thm.conic},  $D(y,x)=\|y-x\|^2/2$, and $g$ is also $L_g$ smooth then by Proposition~\ref{prop.smooth} we obtain the following bound on the relative condition number $L_{f,X,D}/\mu_{f,X,D}$ in terms of the condition number of $g$ and a condition number of the pair $(A,X)$:
\begin{equation}\label{eq.rel.cond.conic}
\frac{L_{f,X,D}}{\mu_{f,X,D}} \le \frac{L_g}{\mu_g} \cdot \left(\|A|\lspan(X)\| \cdot \|(A|X)^{-1}\|\right)^2.
\end{equation}

\subsection{Lower bound on $\mu_{f,X,D}$ when $X$ is a polyhedron}
\label{subsec.poly}

The  results in Section~\ref{subsec.conic} require $X$ to be a convex cone and $A(X)$ to be a linear subspace.  We next provide some results of similar flavor that relax these assumptions in exchange for the assumption that $X$ is a polyhedron.    The crux of the main results in this section is Proposition~\ref{prop.gral}.  This technical result is drawn from the recent paper of Pe\~na, Vera, and Zuluaga~\cite{PenaVZ18}.  The latter paper develops a number of properties of a new class of {\em relative Hoffman bounds}.  In particular, it introduces the sets of tangent cones $\T(X)$ and $\T(A|X)$ described below.  These two sets of tangent cones are at the heart of the main  developments in~\cite{PenaVZ18}.

For a nonempty polyhedron $X\subseteq \R^n$ let $\T(X):=\{T_X(x): x\in X\}$, where $T_X(x)$ is the tangent cone of $X$ at $x$, that is,
\[
T_X(x):=\{d\in \R^n: x + td \in X \; \text{ for some } \; t > 0\}.
\]
We will rely on the following subset of $\T(X)$ that depends on how $A$ and $X$ fit together. Let 
\begin{equation}\label{eq.T.A.X}
\T(A|X):=\{C\in \T(X): A(C) \text{ is a linear subspace and } C \text{ is minimal} \}.
\end{equation}
In this definition, {\em minimal} is to be interpreted as minimal with respect to inclusion.  This restriction guarantees that the set $\T(A|X)$ is of minimal size as it does not include redundant cones from $\T(X)$.

Observe that $\T(X)$ is finite since $X$ is polyhedral and thus $\T(A|X)$ is finite as well. The following example illustrates the interesting relationship between $A$ and the tangent cones of $X$ captured by $\T(A|X)$.

\begin{example}  {\em Suppose $A\in \R^{m\times n}$ and $X = \R^n_+$.  In this case each element of $\T(X)$ is of the form $C_I=\{x\in \R^n: x_I \ge 0\}$ for some $I\subseteq \{1,\dots,n\}$. Observe that $A(C_I)$ is a linear subspace if and only if $Ax=0, \, x_I > 0$ is feasible. Thus the set $\T(A|X)$ is in one-to-one correspondence with the maximal sets $I\subseteq \{1,\dots,n\}$ such that $Ax=0, \, x_I > 0$ is feasible.}
\end{example}

Observe that $\T(A|X) = \{X\}$ when $X$ is a polyhedral cone and $A(X)$ is a linear subspace.  Thus the following proposition subsumes Proposition~\ref{prop.norm} when $X$ is polyhedral.


\begin{proposition}\label{prop.gral} Suppose $\R^n$ and $\R^m$ are endowed with norms. Let $A\in \R^{m\times n}$ and $X \subseteq \R^n$ be a polyhedron such that $A(X)$ contains more than one point.   Then
\begin{equation}\label{eq.inv.gral}
\min_{C\in{  \T(A|X)}} \frac{1}{\|(A|C)^{-1}\|} = \inf_{y,x\in X\atop x\not\in Z_{A,X}(y)} \frac{\|Ay-Ax\|}{\|Z_{A,X}(y)-x\|}.
\end{equation}
\end{proposition}
\begin{proof}
{  This follows as a special case of~\cite[Proposition 5 and Corollary 3]{PenaVZ18}.  }
\end{proof}

\begin{corollary}\label{corol.main}
Suppose $\R^m$ is endowed with the Euclidean norm $\|\cdot\|_2$, $\R^n$ is endowed with a norm $\|\cdot\|$, and $D(x,y) = \frac{1}{2}\|x-y\|^2$.  
If $f(x) = \frac{1}{2} \|Ax - b\|_2^2$ for some $A\in \R^{m\times n}$ and $b\in \R^m$, and 
$X\subseteq \R^n$ is a polyhedron such that $A(X)$ contains more than one point then
\[
 \mu_{f,X,D} = \min_{C\in\T(A|X)} \frac{1}{\|(A|C)^{-1}\|^2}.
\]
\end{corollary}
\begin{proof}
Proceed exactly as in the proof of Corollary~\ref{corol.conic} but apply Proposition~\ref{prop.gral} instead of Proposition~\ref{prop.norm}.
\end{proof}

\begin{theorem}\label{thm.poly} 
Suppose $\R^n$ and $\R^m$ are endowed with norms and   $D(x,y) \le \frac{1}{2}\|x-y\|^2$ for the norm $\|\cdot\|$ in $\R^n$.  Let $A\in\R^{m\times n}, \; g:\R^m \rightarrow \R\cup \{\infty\}$ be a convex differentiable function, and $X\subseteq \R^n$ be a polyhedron such that $A(X)$ contains more than one point.  If $g$ is $\mu_g$ strongly convex for the norm in $\R^m$ 
then the function $f = g \circ A$ satisfies
\[
\mu_{f,X,D} \ge  \min_{C\in\T(A|X)} \frac{\mu_g}{\|(A|C)^{-1}\|^2}.
\]
\end{theorem}
\begin{proof} Proceeding exactly as in the proof of Theorem~\ref{thm.conic} but applying Proposition~\ref{prop.gral} instead of Proposition~\ref{prop.norm} we get
\[
\mu_{f,X,D} \ge \inf_{y,x\in X \atop x\not \in Z_{A,X}(y)} \frac{D_f(y,x)}{\|Z_{A,X}(y) - x\|^2/2} \ge 
\inf_{y,x\in X \atop x\not \in Z_{A,X}(y)} \frac{\mu_g\|Ay-Ax\|^2}{\|Z_{A,X}(y) - x\|^2} =
 \min_{C\in\T(A|X)} \frac{\mu_g}{\|(A|C)^{-1}\|^2}.
\]
\end{proof}

Observe that if $X$ is polyhedral then $\lspan(X-X) \in \T(X)$ and
\[
\|A|\lspan(X-X)\| = \max_{C\in\T(X)} \|A|C\|.
\]
Thus Proposition~\ref{prop.smooth} implies that for $f,X,D$ as in Corollary~\ref{corol.main}, the relative condition $L_{f,X,D}/\mu_{f,X,D}$ has the following expression, which is again strikingly similar to the classical condition number~\eqref{eq.classic.cn} of $f(x) = \frac{1}{2}\|Ax-b\|_2^2$:
\[
\frac{L_{f,X,D}}{\mu_{f,X,D}} = \left(\max_{C\in\T(X)} \|A|C\| \cdot \max_{C\in\T(A|X)} \|(A|C)^{-1}\|\right)^2.
\]
Proposition~\ref{prop.smooth} also implies that if $f,X,D$ are as in Theorem~\ref{thm.poly},  $D(y,x)=\|y-x\|^2/2$, and $g$ is $L_g$ smooth then the  relative condition  number $L_{f,X,D}/\mu_{f,X,D}$ can be bounded in terms of the condition number of $g$ and a condition number of the pair $(A,X)$ as follows:
\begin{equation}\label{eq.rel.cond.poly}
\frac{L_{f,X,D}}{\mu_{f,X,D}} \le \frac{L_g}{\mu_g} \cdot \left(\max_{C\in\T(X)} \|A|C\| \cdot \max_{C\in\T(A|X)} \|(A|C)^{-1}\|\right)^2.
\end{equation}

We next place some of the developments by Pe\~na and Rodr\'iguez~\cite{PenaR16} in the context of this paper.  To that end, consider the special case when $X$ is the standard simplex $\Delta_{n-1}:=\{x\in\R^n_+: \|x\|_1 = 1\}$ in $\R^n$.  For $A = \matr{a_1 & \cdots & a_n} \in \R^{m\times n}$ let $\conv(A) := \conv(\{a_1,\dots,a_n\}) = \{Ax: x\in \Delta_{n-1}\}$ and let  $\faces(\conv(A))$ denote the set of faces of $\conv(A)$.  Furthermore, for $F \in \faces(\conv(A))$ let $A\setminus F$ denote the set of columns of $A$ that do not belong to $F$.  Suppose $\R^m$ is endowed with a norm and for $F,G\subseteq \R^m$ let $\dist(F,G):=\inf_{u\in F, v\in G} \|u-v\|$.  Following~\cite{PenaR16} define the {\em facial distance} $\Phi(A)$ of $A$ as follows
\begin{equation}\label{eq.facial.dist.def}
\Phi(A):=\dmin_{F \in \faces(\conv(A))\atop \emptyset \ne F \ne \conv(A)} \dist(F,\conv(A\setminus F)).
\end{equation}
Let $\diam(A)$ denote the {\em diameter} of the set of columns of $A$ defined as follows
\begin{equation}\label{eq.diam.def}
\diam(A):=\max_{i,j\in \{1,\dots,n\}} \|a_i - a_j\|.
\end{equation}
In the special case when $X = \Delta_{n-1}$ it follows from~\cite[Theorem 1]{PenaR16} that~\eqref{eq.inv.gral} in Proposition~\ref{prop.gral} has the following geometric characterization
\begin{equation}\label{eq.facial.dist}
\min_{y,x\in \Delta_{n-1} \atop x\not\in Z_{A,X}(y)}\frac{\|Ay-Ax\|}{\|Z_{A,X}(y) - x\|_1} = \frac{\Phi(A)}{2}.
\end{equation}
Furthermore, in this same special case when $X = \Delta_{n-1}$ it is easy to see that~\eqref{eq.norm.map} has the following geometric characterization
\begin{equation}\label{eq.diam}
\max_{x,y\in \Delta_{n-1} \atop x\ne y}\frac{\|Ay-Ax\|}{\|y - x\|_1} = \frac{\diam(A)}{2}.
\end{equation}

Figure~\ref{fig.facial.diam} gives a visualization of $\conv(A)$ and of
the facial distance $\Phi(A)$ for $A = I_3$ and $A = I_4$.  It depicts $\conv(A)$ and $\Phi(A)$ in the hyperplane $\{x: \ip{\1}{x} = 1\}$.

\begin{figure}
\begin{center}\includegraphics[width = 10cm]{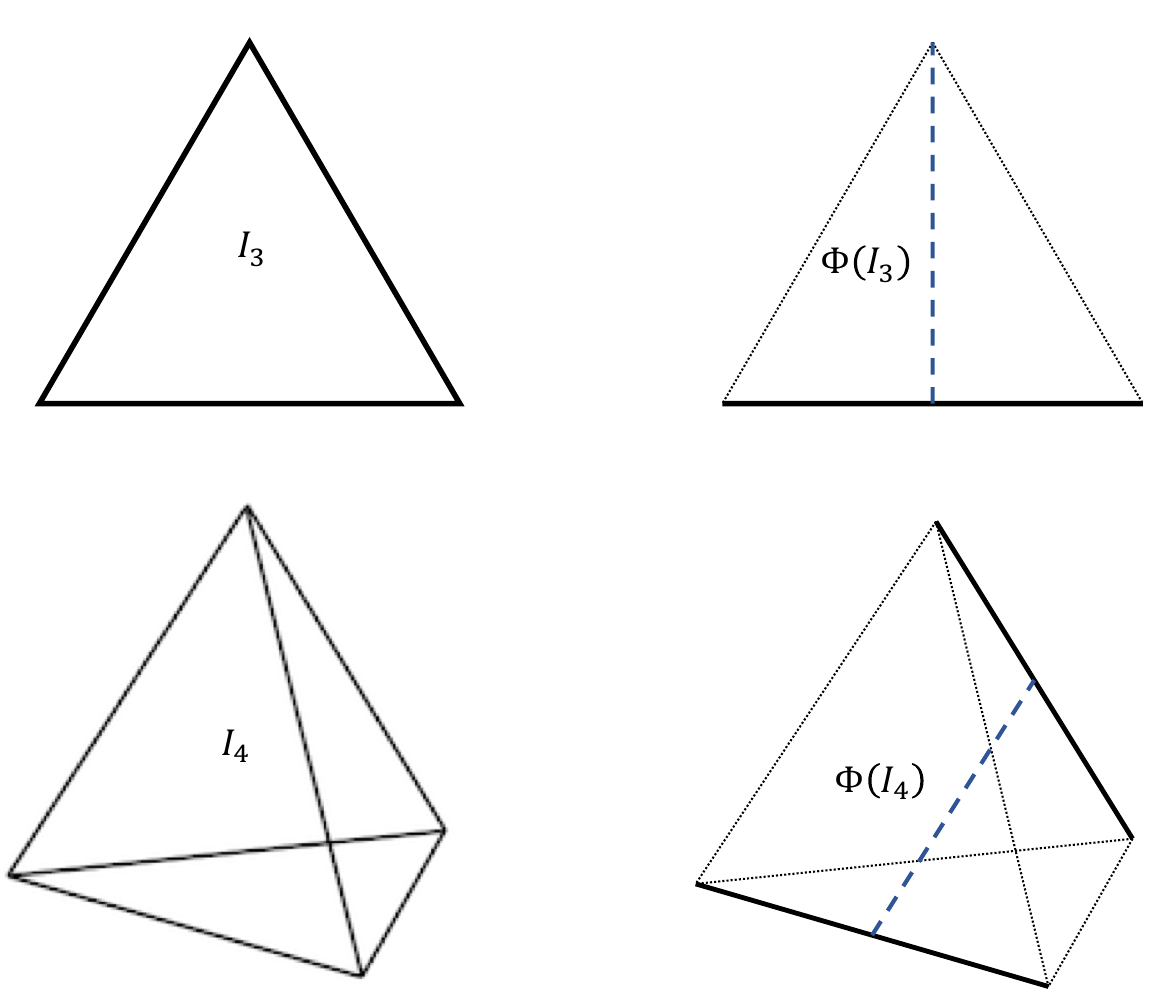}\end{center}
\caption{Depiction of $\conv(A)$ and $\Phi(A)$ for  $A = I_3$ and $A = I_4$.}
\label{fig.facial.diam}
\end{figure}


Example~\ref{ex.simplex} below, a special case of Corollary~\ref{corol.main}, shows that for $f(x) = \frac{1}{2}\|Ax-b\|_2^2$, $X = \Delta_{n-1}$, and $D(y,x) = \frac{1}{2}\|y-x\|_1^2$ the relative condition number  $L_{f,\Delta_{n-1},D}/\mu_{f,\Delta_{n-1},D}$ is the square of $\diam(A)/\Phi(A)$, which has a flavor of an aspect ratio of $\conv(A)$.  This gives an interesting analogy to~\eqref{eq.classic.cn}.

\begin{example}\label{ex.simplex} {\em Suppose $\R^n$ is endowed with the $\ell_1$ norm, $\R^m$ is endowed with the Euclidean $\ell_2$ norm, and $f(x) = \frac{1}{2} \|Ax - b\|_2^2$ for some $A\in \R^{m\times n}$ with at least two different columns and $b\in \R^m$.  Then for $D(y,x) := \frac{1}{2}\|y-x\|_1^2$ Corollary~\ref{corol.main} and identities~\eqref{eq.diam} and~\eqref{eq.facial.dist} yield
\[
L_{f,\Delta_{n-1},D} = \frac{\diam(A)^2}{4} \; \text{ and } \; \mu_{f,\Delta_{n-1},D} = \frac{\Phi(A)^2}{4}.
\]
In particular,
\[
\frac{L_{f,\Delta_{n-1},D}}{\mu_{f,\Delta_{n-1},D}} = \left(\frac{\diam(A)}{\Phi(A)}\right)^2.
\]
{  
More generally, if $f(x) = g(Ax)$ for some $L_g$ smooth and $\mu_g$ strongly convex function $g$ then
\[
L_{f,\Delta_{n-1},D} \le \frac{L_g\cdot \diam(A)^2}{4} \; \text{ and } \; \mu_{f,\Delta_{n-1},D} \ge  \frac{\mu_g\cdot\Phi(A)^2}{4}.
\]
In particular,
\[
\frac{L_{f,\Delta_{n-1},D}}{\mu_{f,\Delta_{n-1},D}} \le \frac{L_g}{\mu_g}\cdot \left(\frac{\diam(A)}{\Phi(A)}\right)^2.
\]
}
}
\end{example}

\section{Properties of $\mu_{f,X,D}^\star,$ and $\mu_{f,X,D}^\sharp$}
\label{sec.prop.qrel}

 We next provide bounds on $\mu_{f,X,D}^\star$ and $\mu_{f,X,D}^\sharp$ analogous to those developed in Section~\ref{sec.prop.rel} for $\mu_{f,X,D}$.  
Proposition~\ref{prop.qrel} already established  $\mu_{f,X,D}^\star \ge \mu_{f,X,D}\ge 0$.  It is intuitively clear that $\mu_{f,X,D}^\star$ could be a lot larger.  When $D$ is a squared norm, the exact same technique used in~\cite[Theorem 1]{NecoNG18} show that $\mu_{f,X,D}^\sharp \ge \mu_{f,X,D}^\star$.  Indeed, when $D$ is a squared norm, the relationship among other variants of strong convexity introduced~\cite{NecoNG18} extend to our context in a straightforward fashion as we next explain.

\begin{definition}{\em Suppose $(f,X,D)$ satisfy Assumption~\ref{assump.min}.
\begin{itemize}
\item[(a)] We say that $f$ has {\em $D$-under approximation} on $X$ if there exists a constant $\mu > 0$ such that
\begin{equation}\label{eq.quad.under.approx}
D_f(x,\bar x) \ge \mu D(\bar x, x) \; \text{ for all } \; x\in X.
\end{equation}
\item[(b)] We say that $f$ has {\em $D$-gradient growth} on $X$ if there exists a constant $\mu > 0$ such that
\begin{equation}\label{eq.grad.growth}
\ip{\nabla f(x) - \nabla f(\bar x)}{x-\bar x} \ge \mu D(\bar x,x) \; \text{ for all } \; x\in X.
\end{equation}
\end{itemize}
}
\end{definition}

Suppose $(f,X,D)$ satisfies Assumption~\ref{assump.min} and $D$ is a squared norm.  Then  for $\mu>0$~\cite[Theorem 4]{NecoNG18} yields the following chain of implications for $(f,X,D,\mu)$: 
\[ \eqref{eq.strong.conv.rel} \Rightarrow \eqref{eq.quasi.strong.convex.rel} \Rightarrow
\eqref{eq.quad.under.approx} \Rightarrow \eqref{eq.grad.growth} \Rightarrow \eqref{eq.func.growth}.\]
We note that~\cite[Theorem 4]{NecoNG18} is stated and proven for the Euclidean norm but the same statement and proof hold for any norm. 

From the above chain of implications it follows that if $(f,X,D)$ satisfies Assumption~\ref{assump.min} and $D$ is a squared norm then $\mu_{f,X,D} \le \mu^\star_{f,X,D} \le \mu^\sharp_{f,X,D}$. 
In particular, any lower bound on $\mu_{f,X,D}$, such as those  in Theorem~\ref{thm.conic} or Theorem~\ref{thm.poly}, is also a lower bound on $\mu^\star_{f,X,D}$ and on $\mu^\sharp_{f,X,D}$ when $D$ is a squared norm.  We next show that the ideas in Section~\ref{sec.prop.rel} can be extended to obtain sharper bounds on these two constants.
\subsection{A sharper lower bound on $\mu^\star_{f,X,D}$}
\label{subsec.sharp}
Suppose $A\in \R^{m\times n}$ and $X\subseteq \R^n$ is a polyhedron such that $A(X)$ contains more than one point, and  $S\subseteq X$ is nonempty.  Proposition~\ref{prop.gral} readily implies 
\begin{equation}\label{eq.bound.new}
\inf_{  y\in S,\, x\in X \atop x\not\in Z_{A,X}(y)} \frac{\|Ay-Ax\|}{\|Z_{A,X}(y)-x\|}
\ge 
\min_{C\in\T(A|X)} \frac{1}{\|(A|C)^{-1}\|} > 0.
\end{equation}
Proposition~\ref{prop.gral.local} below, which extends Proposition~\ref{prop.gral}, gives a sharper version of~\eqref{eq.bound.new}.
Suppose $A\in \R^{m\times n}, \; X\subseteq \R^n$ is a polyhedron, and  $S\subseteq X$ is nonempty.  Let 
\begin{equation}\label{eq.T.A.X.S}
{  \T(A|X,S):=\{T_X(x;A,S): x\in X\}}
\end{equation}
where
\[
T_X(x;A,S):= \{d\in \R^n: x+td \in X \, \text{ and } \, A(x+td) \in {\conv}(A(S)) \; \text{ for some } \; t > 0\}.
\]
   
Proposition~\ref{prop.gral.local} can be proven via a straightforward modification of techniques in~\cite{PenaVZ18}.  We provide the details of this modification in Appendix~\ref{appendix}.

\begin{proposition}\label{prop.gral.local} 
Suppose $\R^n$ and $\R^m$ are endowed with norms. 
Let $A\in \R^{m\times n}$ and $X \subseteq \R^n$ be a polyhedron such that $A(X)$ contains more than one point.   Then for all nonempty $S\subseteq X$ 
\begin{equation}\label{eq.sharp}
 \inf_{  y\in S,\, x\in X \atop x\not\in Z_{A,X}(y)} \frac{\|Ay-Ax\|}{\|Z_{A,X}(y)-x\|} \ge \inf_{C\in\T(A|X,S)} \frac{1}{\|(A|C)^{-1}\|} \ge \min_{C\in\T(A|X)} \frac{1}{\|(A|C)^{-1}\|}.
\end{equation}
Furthermore, if $A(S)$ is  convex then 
\begin{equation}\label{eq.prop.local}
 \inf_{  y\in S,\, x\in X \atop x\not\in Z_{A,X}(y)} \frac{\|Ay-Ax\|}{\|Z_{A,X}(y)-x\|} = \inf_{C\in\T(A|X,S)} \frac{1}{\|(A|C)^{-1}\|}.
\end{equation}
\end{proposition}

{  
\begin{corollary}\label{corol.local}
Suppose $\R^m$ is endowed with the Euclidean norm $\|\cdot\|_2$, $\R^n$ is endowed with a norm $\|\cdot\|$, and $D(x,y) = \frac{1}{2}\|x-y\|^2$.  
If $f(x) = \frac{1}{2} \|Ax - b\|_2^2$ for some $A\in \R^{m\times n}$ and $b\in \R^m$, and 
$X\subseteq \R^n$ is a polyhedron such that $A(X)$ contains more than one point and $X^\star:=\argmin_{x\in X} f(x)\ne \emptyset$.  Then
\[
 \mu_{f,X,D}^\star = \inf_{C\in\T(A|X,X^\star)} \frac{1}{\|(A|C)^{-1}\|^2}.  
\]
\end{corollary}
\begin{proof}
Proceed exactly as in the proof of Corollary~\ref{corol.conic} but apply Proposition~\ref{prop.gral.local} instead of Proposition~\ref{prop.norm}.
\end{proof}
}

The following theorem gives a lower bound on $\mu^\star_{f,X,D}$ analogous to the one on $\mu_{f,X,D}$ in Theorem~\ref{thm.poly}.  In light of Proposition~\ref{prop.gral.local}, the lower bound on $\mu^\star_{f,X,D}$ in Theorem~\ref{thm.main.local} is at least as large, and possibly much larger, than the one on $\mu_{f,X,D}$ in Theorem~\ref{thm.poly}.

\begin{theorem}\label{thm.main.local}
Suppose $\R^n$ and $\R^m$ are endowed with norms and   $D(y,x) \le \frac{1}{2}\|y-x\|^2$ for the norm $\|\cdot\|$ in $\R^n$. Let $A\in\R^{m\times n}, \; g:\R^m \rightarrow \R\cup \{\infty\}$ and $X\subseteq \R^n$ be a polyhedron such that $A(X)$ has more than one point.  If $g$ is $\mu_g$-strongly convex for the norm in $\R^m$ then the function $f = g \circ A$ satisfies
\[
\mu_{f,X,D}^\star \ge 
\inf_{C\in\T(A|X,X^\star)} \frac{\mu_g}{\|(A|C)^{-1}\|^2}.
\]
\end{theorem}
\begin{proof} 
Observe that for all $y\in X^\star$ and $x\in X$
\[
D_f(y,x) = g(Ay) - g(Ax) - \ip{g(Ax)}{A(y-x)}. 
\]
Since $g$ is $\mu_g$ strongly convex on $A(X)$, it follows that $D_f(y,x) \ge \mu_g\|Ay-Ax\|^2/2$ for all $y \in X^\star$ and $x\in X$,
 and it also follows that $Z_{A,X}(y) = \{x\in X: Ax = Ay\} = X^\star$ for all $y\in X^\star$. Therefore 
\[
\mu_{f,X,D}^\star \ge \inf_{x\in X\setminus X^\star} \frac{D_f(\bar x,x)}{\|\bar x - x\|^2/2}  \ge \inf_{y\in X^\star \atop x\in X\setminus X^\star} \frac{D_f(y,x)}{\|y - x\|^2/2} \ge 
 \inf_{  y\in X^\star,\, x\in X \atop x\not\in Z_{A,X}(y)} \frac{\mu_g\|Ay-Ax\|^2}{\|Z_{A,X}(y) - x\|^2}. \]
To finish, 
apply Proposition~\ref{prop.gral.local}.
\end{proof}

Once again there is an interesting connection with the developments in~\cite{PenaR16} when $X = \Delta_{n-1}$.  Consider the special case when $X = \Delta_{n-1}, \; A \in \R^{m\times n}$ has at least two different columns, $S\subseteq \Delta_{n-1}$ is nonempty, and $G\in \faces(\conv(A))$ is the smallest face of $\conv(A)$ that contains $A(S)$.  From~\cite[Theorem 3]{PenaR16} it follows that if $\R^n$ is endowed with the one-norm then 
\begin{equation}\label{eq.facial.dist.local}
\inf_{y\in S, x\in X \atop x\not\in Z_{A,X}(y)} \frac{\|Ay-Ax\|}{\|Z_{A,X}(y)-x\|_1} \ge \min_{F\in \faces(G)\atop \emptyset \ne F\ne \conv(A)}
\dist(F,\conv(A\setminus F)).
\end{equation}
The following example illustrates the difference between $\mu_{f,X,D}$ and $\mu^\star_{f,X,D}$. 

\begin{example}{\em
Suppose $\R^n$ is endowed with the one-norm and $D(y,x) := \frac{1}{2}\|y-x\|_1^2$.  Suppose $\R^m$ is endowed with the Euclidean norm, and $f(x) = \frac{1}{2} \|Ax - b\|_2^2$ for some $A\in \R^{m\times n}$ with at least two different columns and $b\in \R^m$.  As noted in Example~\ref{ex.simplex}, in this case
\[
\mu_{f,\Delta_{n-1},D} = \frac{\Phi(A)^2}{4} = \frac{1}{4} \left(\min_{F\in\faces(\conv(A))\atop \emptyset\ne F\ne \conv(A)}\dist(F,\conv(A\setminus F)) \right)^2.
\]
This relative strong convexity constant depends only on $A$ but not on $b$.  On the other hand, the smallest face of $\conv(A)$ containing $X^\star$ is 
\[
G(b)  := \argmin_{G\in\faces(\conv(A))} \dist(G,b),
\]
which evidently depends on both $A$ and $b$.  Theorem~\ref{thm.main.local} and~\eqref{eq.facial.dist.local} yield
\[
\mu^\star_{f,\Delta_{n-1},D} \ge \frac{1}{4} \left(\min_{F\in\faces(G(b))\atop \emptyset\ne F\ne \conv(A)}\dist(F,\conv(A\setminus F)) \right)^2.
\]
It is evident that
\[
\min_{F\in\faces(G(b))\atop \emptyset\ne F\ne \conv(A)}\dist(F,\conv(A\setminus F)) \ge \Phi(A).
\]
Furthermore, as it is illustrated in~\cite{PenaR16}, the difference between these two quantities can be arbitrarily large.  Consequently, the bound in Theorem~\ref{thm.main.local} can be far sharper than that in Theorem~\ref{thm.poly}. 
}
\end{example}

\subsection{A sharper lower bound on $\mu^\sharp_{f,X,D}$}

Suppose $f:\R^n\rightarrow\R\cup\{\infty\}$ is defined as  $f(x) = g(Ax) + \ip{c}{x}$ where $g:\R^m\rightarrow \R\cup\{\infty\}$ is a strongly convex function, 
 $A\in \R^{m\times n}$ and $c\in \R^n$.  Theorem~\ref{thm.main.local} does not apply to this kind of function due to the extra linear term  $\ip{c}{x}$.  Indeed for a function of this form the constant $\mu_{f,X,D}^\star$ may be zero, see Example~\ref{ex.counter.2} below.  On the other hand, the next result shows that for a function of this form and for a  polyhedral set $X$ it is always the case that $\mu^\sharp_{f,X,D}>0$ provided a suitable linear cut is added to $X$.
\begin{theorem}\label{thm.main.growth} 
Suppose $\R^n$ and $\R^m$ are endowed with norms and $D(x,y)   \le \frac{1}{2}\|x-y\|^2$ for the norm $\|\cdot\|$ in $\R^n$. Let $A\in\R^{m\times n}, \; c\in \R^n,$ and $X\subseteq \R^n$ be a polyhedron such that $A(X)$ contains more than one point. Suppose $g:\R^m \rightarrow \R\cup \{\infty\}$ is $\mu_g$-strongly convex for the norm in $\R^m$ and $f:\R^n \rightarrow \R\cup \{\infty\}$ is defined via $f(x) = g(Ax) + \ip{c}{x}$.  Then the vector $v := 2\nabla f(y)$ is the same for all $y\in X^\star$ and satisfies $\ip{v}{x-y} \ge 0$ for all $x\in X, \; y\in X^\star$.  Furthermore, one of the following two possible cases applies depending on the range of values of $\ip{v}{x-y}$ for $x\in X, y\in X^\star$.

\begin{description}
\item[Case 1:] For all $x\in X, y\in X^\star$ we have $\ip{v}{x-y} = 0$.  In this case
\[
\mu^\sharp_{f,X,D} \ge
\inf_{C\in\T(A|X,X^\star)} \frac{\mu_g}{\|(A|C)^{-1}\|^2}. 
\]
\item[Case 2:] For some $x\in X, y\in X^\star$ we have $\ip{v}{x-y}>0$.  In this case for all $\delta > 0$ 
\begin{align*}
\mu^\sharp_{f,X_\delta,D} 
\ge 
 \inf_{C\in \T(M|X_\delta,X^\star)} \frac{1}{\|(M|C)^{-1}\|^2},
\end{align*}
for the polyhedron $X_\delta:=\{x\in X: \ip{v}{x-y} \le \delta \text{ for all } y\in X^\star\} \supseteq X^\star$, the matrix $M\in \R^{(m+1)\times n},$ and the norm $\|\cdot\|$ in $\R^{m+1}$  defined as follows
\[
M:=\matr{\sqrt{\mu_g}\cdot A \\ \frac{1}{\sqrt{\delta}} \cdot v\transp} \; \text{ and } \; \left\|\matr{y\\y_{m+1}}\right\|:= \sqrt{\|y\|^2 + y_{m+1}^2}. 
\]
\end{description}
\end{theorem}
\begin{proof}
The optimality conditions for $\min_{x\in X}f(x)$ imply that 
\begin{equation}\label{eq.opt.conds}
\ip{\nabla f(y)}{x - y} =  
\ip{A\transp\nabla g(Ay) + c}{x-y} \ge 0 \text{ for all } x\in X, y\in Y^\star.
\end{equation}
Thus for all $y,y'\in X^\star$ the strong convexity of $g$ and \eqref{eq.opt.conds} imply
\[
\mu_g\|Ay-Ay'\|^2 \le \ip{\nabla g(Ay) - \nabla g(Ay')}{Ay-Ay'} = \ip{\nabla f(y) - \nabla f(y')}{y-y'} \le 0.
\]
Hence $Ay=Ay'$ whenever $y,y'\in X^\star.$ In particular, $v = 2\nabla f(y) = 2(A\transp \nabla g(Ay) + c)$ is the same for all $y\in X^\star$.  Furthermore, the optimality conditions for $\min_{x\in X} f(x)$ imply that $\ip{v}{x-y} \ge 0$ for all $x\in X, y\in Y^\star$.  In particular, $\ip{v}{y} = \min_{x\in X} \ip{v}{x}$ for all $y\in X^\star$.

Next,  the strong convexity of $g$ on $A(X)$ implies that for all $x\in X, y\in X^\star$
\begin{align*}
f(x) - f^\star &= g(Ax) - g(Ay) + \ip{c}{x-y} \\
& \ge \frac{\mu_g}{2} \|Ax-Ay\|^2 + \ip{\nabla g(Ay)}{Ax-Ay} +  \ip{c}{x-y} \\
& = \frac{1}{2}\left( \mu_g \|Ax-Ay\|^2 + \ip{v}{x-y} \right).
\end{align*}
If $\ip{v}{x-y} = 0$ for all $x\in X,y\in X^\star$ then Case 1 applies.  In this case $Z_{A,X}(y) = \{x\in X: Ax = Ay\} = X^\star$ for all $y\in X^\star$ and thus
\[
\mu^\sharp_{f,X,D} \ge \inf_{y\in X^\star\atop x \in X\setminus X^\star}\frac{f(x) - f^\star}{\|y-x\|^2/2} \ge 
\inf_{y\in X^\star, \, x \in X
\atop x\not\in Z_{A,X}(y)} \frac{\mu_g  \|Ay - Ax\|^2}{\|Z_{A,X}(y)-x\|^2}.
\]
If $\ip{v}{x-y} > 0$ for some $x\in X,y\in X^\star$ then Case 2 applies.  In this case $Z_{M,X}(y) = \{x\in X: Ax = Ay, \ip{v}{x} = \ip{v}{y}\} = X^\star$ for all $y\in X^\star$ and thus
\[
\mu^\sharp_{f,X_\delta,D} \ge \inf_{y\in X^\star\atop x \in X_\delta\setminus X^\star}\frac{f(x) - f^\star}{\|y-x\|^2/2} \ge \inf_{  y\in X^\star, \, x\in X_\delta\atop x\not\in  Z_{M,X}(y)} \frac{\mu_g  \|Ay - Ax\|^2 + \ip{v}{y-x}}{\|Z_{M,X}(y)-x\|^2}.
\]
Next, observe that for $y\in X^\star$ and $x\in X_\delta$ 
\[
\mu_g\|Ay - Ax\|^2 + \ip{v}{y-x} \ge \mu_g\|Ay - Ax\|^2 + \frac{\ip{v}{y-x}^2}{\delta} = \|My-Mx\|^2.
\]
To finish, apply Proposition~\ref{prop.gral.local} in either case.
\end{proof}

Observe that if $X$ in Theorem~\ref{thm.main.growth} is bounded then Case 2 gives a lower bound on $\mu^\sharp_{f,X,D}$ by taking $\delta:=\max_{x\in X,y\in X^\star} \ip{v}{x-y}$ because $X = X_\delta$ for this choice of $\delta$.

We conclude this section with a simple example showing that $\mu^\sharp_{f,X,D} > \mu^\star_{f,X,D} = 0$ can occur.  The example also shows that the additional bound on $X_\delta$ in Theorem~\ref{thm.main.growth}, Case 2 cannot simply dropped without making some additional assumptions.

\begin{example}\label{ex.counter.2}
{\em Let $\R^3$ be endowed with the one-norm and let $D(y,x):=\frac{1}{2}\|y-x\|_1^2.$ Suppose $f:\R^3\rightarrow \R$ is as follows
$$f(x) = \frac{1}{2} (x_1-x_2)^2 + x_3.$$

If $X:=\Delta_2\subseteq \R^3$ then $X^\star = \{\matr{1/2& 1/2 & 0}\transp\}$.  For $x =  \matr{0 & 0 & 1}\transp$ we have $f(\bar x) - f(x) - \ip{\nabla f(x)}{\bar x-x} = 0$ and $\|\bar x-x\|_1=2$.  Hence $\mu^\star_{f,X,D} = 0$.  On the other hand, Theorem~\ref{thm.main.growth} implies that $
\mu^\sharp_{f,X,D} > 0.$  A more careful calculation shows that in this case $\mu^\sharp_{f,X,D} = 1/2$.

On the other hand, if $X = \R^3_+$ then $X^\star = \{\matr{t & t& 0}\transp: t\ge 0\}.$  For $t>0$ and $x = \matr{0 & 0 & t}\transp$ we have $f(x) - f^\star = t$ and $\|X^\star - x\|_1 = t$.  Therefore $\mu^\sharp_{f,X,D} = 0$.  Furthermore, in the context of Theorem~\ref{thm.main.growth} we have $v = \matr{0 & 0 & 2}\transp$.  A simple calculation shows that  for all $\delta > 0$ we have $X_\delta=\{x\in X: x_3 \le \delta/2\}$ and $\mu^\sharp_{f,X_\delta,D} = 2/(2+\delta/2).$
}
\end{example}
\section{Convergence of first-order methods}
\label{sec.algos}

This section details linear convergence results for the mirror descent algorithm, Frank-Wolfe algorithm, and Frank-Wolfe algorithm with away steps for problem~\eqref{eq.Prob}.  The linear convergence statements for the three algorithms are strikingly similar.  They are stated in terms of the relative constants $L_{f,X,D}$ and $\mu^\star_{f,X,D}, \mu^\sharp_{f,X,D}$ for suitable choices of distance-like functions $D$.

\subsection{Mirror descent algorithm}

Suppose $h:\R^n\rightarrow \R\cup\{\infty\}$ is convex and differentiable on $X\subseteq \R^n$ and the {\em Bregman proximal map}
\[
g\mapsto \argmin_{y\in X} \{\ip{g}{y} + L D_h(y,x)\}
\]
is computable for $x\in X$ and $L >0$.  The {\em mirror descent algorithm} for problem~\eqref{eq.Prob} is based on the following update for $x\in X$:
\[
x_+:=\argmin_{y\in X} \{\ip{\nabla f(x)}{y} + L D_h(y,x)\}.
\] 
Algorithm~\ref{algo.mirror} gives a description of the mirror descent algorithm for \eqref{eq.Prob}.  

\begin{algorithm}
  \caption{Mirror descent algorithm
    \label{algo.mirror}}
  \begin{algorithmic}[1]
\State Pick $x_0 \in X$ ;
 \For{$k=0,1,2,\dots$}
 \State choose $L_k > 0$ 
\State $x_{k+1} = \dargmin_{y\in X}\left\{ \ip{\nabla f(x_k)}{y} + L_k D_h(y,x_k)\right\}$
\EndFor
  \end{algorithmic}
\end{algorithm}

Proposition~\ref{prop.lin.mirror} and Proposition~\ref{prop.lin.mirror.2} show the linear convergence of Algorithm~\ref{algo.mirror} provided that  suitable relative smoothness and relative quasi-strong convexity or relative functional growth conditions hold.   Throughout the remaining of this subsection we assume that $(f,X,D_h)$ satisfy Assumption~\ref{assump.blanket}.

We should note that Proposition~\ref{prop.lin.mirror} and its proof are straightforward modifications of the  linear convergence results in~\cite{LuFN18,Tebo18}.    However, Proposition~\ref{prop.lin.mirror} shows that the linear convergence of Algorithm~\ref{algo.mirror} holds with a sharper rate and under more general assumptions than those in~\cite{LuFN18,Tebo18}.   In particular, the rate in Proposition~\ref{prop.lin.mirror} is stated in terms of a relative quasi-strong convexity constant, which is always at least as large and possibly much larger than the kind of  relative strong convexity constant in~\cite{LuFN18,Tebo18}.
Furthermore, our results in Section~\ref{sec.prop.rel} and Section~\ref{sec.prop.qrel} guarantee linear convergence when $f$ is of the form $g\circ A$ provided $g$ and $h$ satisfy smoothness and strong convexity assumptions.  The linear convergence results in~\cite{LuFN18,Tebo18} do not apply for functions of this form because they are not strictly convex and thus the kind of relative strong convexity constant in~\cite{LuFN18,Tebo18} is typically zero.

\medskip

 The following lemma, which is a straightforward extension of results presented in~\cite{Tebo18}, provides the crux of the proof of Proposition~\ref{prop.lin.mirror}.

\begin{lemma}\label{lemma.crux.mirror}
 Suppose $L:=L_{f,X,D_h} < \infty$ and $\mu:=\mu_{f,X,D_h}^\star > 0$.
 If $x\in X$ and
 \begin{equation}\label{eq.grad.step}
 x_+ = \argmin_{y\in X}\{f(x)+\ip{\nabla f(x)}{y-x} + L D_h(y,x)\}
 \end{equation}
then
 \begin{equation}\label{eq.claim}
f(x_+) - f^\star \le (L-\mu)D_h(\bar x,x) -  LD_h(\bar x,x_+).
\end{equation}
\end{lemma}
\begin{proof}
Since $L=L_{f,X,D_h}$ and $\mu=\mu_{f,X,D_h}^\star$ we have
\begin{equation}\label{eq.step.1}
f(x_+) \le f(x) + \ip{\nabla f(x)}{x_+-x} + L D_h(x_+,x).
\end{equation}
and
\begin{equation}\label{eq.step.2}
f(x) \le f^\star + \ip{\nabla f(x)}{x-\bar x} - \mu D_h(\bar x,x).
\end{equation}
In addition, the three-point property of $D_h$~\cite[Lemma 3.1]{ChenT93} yields
\begin{equation}\label{eq.three.point}
D_h(x_+,x) = D_h(\bar x,x) - D_h(\bar x, x_+) + \ip{\nabla h(x_+) - \nabla h(x)}{x_+-\bar x}.
\end{equation}
By putting together~\eqref{eq.step.1},~\eqref{eq.step.2}, and~\eqref{eq.three.point} we get
\begin{align*}
f(x_+) \le  f^\star & + (L - \mu) D_h(\bar x, x) - L D_h(\bar x,x_+) \\ &
+ \ip{\nabla f(x)+L (\nabla h(x_+) - \nabla h(x))}{x_+-\bar x}.
\end{align*}
We get~\eqref{eq.claim} by observing that the optimality conditions for~\eqref{eq.grad.step} imply 
\[
\ip{\nabla f(x)+L (\nabla h(x_+) - \nabla h(x))}{x_+-\bar x} \le 0.
\]
\end{proof}

\begin{proposition}\label{prop.lin.mirror}  
  Suppose $L:=L_{f,X,D_h} < \infty$ and $\mu:=\mu_{f,X,D_h}^\star > 0$.
If $L_k = L, \; k=0,1,\dots$ in Algorithm~\ref{algo.mirror}  then the iterates generated by Algorithm~\ref{algo.mirror} satisfy
\begin{equation}\label{eq.lin.conv}
D_h(X^\star,x_{k}) \le \left(1-\frac{\mu}{L} \right)^kD_h(X^\star,x_0) \; \text{ for } \; k=0,1,\dots
\end{equation}
and
\[
f(x_k) - f^\star \le L\left(1-\frac{\mu}{L} \right)^kD_h(X^\star,x_0) \; \text{ for } \; k=1,2,\dots.
\]
\end{proposition}
\begin{proof}

Lemma~\ref{lemma.crux.mirror} applied to $x = x_k$ implies that
\begin{equation}\label{eq.inductive}
(L-\mu)D_h(\bar x_k,x_k) - LD_h(\bar x_k,x_{k+1})  \ge f(x_{k+1}) - f^\star \ge 0 \; \text{ for } \; k=0,1,\dots.
\end{equation}
Therefore
\[
D_h(X^\star,x_{k+1}) \le D_h(\bar x_k, x_{k+1}) \le \left(1-\frac{\mu}{L} \right)D_h(\bar x_k,x_k) 
 \; \text{ for } \; k=0,1,\dots.
\]
Thus~\eqref{eq.lin.conv} readily follows. Inequality~\eqref{eq.inductive} also yields
\[
f(x_k) - f^\star \le L \left(1-\frac{\mu}{L} \right) D_h(X^\star,x_{k-1}) \le L \left(1-\frac{\mu}{L} \right)^k D_h(X^\star,x_0) \; \text{ for } \; k=1,2,\dots.
\]
\end{proof}

Proposition~\ref{prop.lin.mirror} implies that if $L:=L_{f,X,D_h} < \infty$ and $\mu:=\mu_{f,X,D_h}^\star > 0$ then Algorithm~\ref{algo.mirror} yields $x_k \in X$ such that $f(x_k) - f^\star < \epsilon$ in at most
\[
\Oh\left(\frac{L}{\mu} \log\left(\frac{LD_h(X^\star,x_0)}{\epsilon}\right)\right)
\]
iterations.  

Proposition~\ref{prop.lin.mirror.2} below shows that the same kind of iteration bound holds under a relative functional growth assumption instead of the quasi strong convexity assumption in Proposition~\ref{prop.lin.mirror}.    We note that although Proposition~\ref{prop.lin.mirror.2} is similar in flavor to Proposition~\ref{prop.lin.mirror}, it is stated in terms of the novel concept of relative functional growth.  Furthermore, neither Proposition~\ref{prop.lin.mirror} nor Proposition~\ref{prop.lin.mirror.2} implies the other since neither 
$\mu_{f,X,D_h}^\star$ nor $\mu_{f,X,D_h}^\sharp$  necessarily bounds the other.  (See Example~\ref{ex.counter.1} and Example~\ref{ex.counter.2}.)

\begin{proposition}\label{prop.lin.mirror.2}  
 Suppose $L:=L_{f,X,D_h} < \infty$ and $\mu:=\mu_{f,X,D_h}^\sharp > 0$.
If $L_k = L, \; k=0,1,\dots$ in Algorithm~\ref{algo.mirror}  then for $K = \lceil 2L/\mu \rceil$ the iterates generated by Algorithm~\ref{algo.mirror} satisfy
\begin{equation}\label{eq.lin.conv.2}
D_h(X^\star,x_{k+K}) \le \frac{D_h(X^\star,x_{k})}{2}  \; \text{ for } \; k=0,1,2,\dots.
\end{equation}
In addition,  Algorithm~\ref{algo.mirror} yields $x_k \in X$ such that $f(x_k) - f^\star < \epsilon$ in at most
\begin{equation}\label{eq.lin.conv.3}
\Oh\left(\frac{L}{\mu} \log\left(\frac{LD_h(X^\star,x_0)}{\epsilon}\right)\right)
\end{equation}
iterations.
\end{proposition}  
\begin{proof} Since $L_k = L = L_{f,X,D_h}$, it follows from~\cite[Theorem 3.1]{LuFN18} that the $(k+K)$-th iterate generated by Algorithm~\ref{algo.mirror} satisfies
\[
f(x_{k+K}) - f^\star \le \frac{L}{K}D_h(X^\star,x_k). 
\]
Therefore, since $\mu:=\mu_{f,X,D_h}^\sharp > 0$,
\[
D_h(X^\star,x_{k+K}) \le \frac{1}{\mu}(f(x_{k+K}) - f^\star) \le \frac{L}{\mu K}D_h(X^\star,x_k) \le \frac{D_h(X^\star,x_k)}{2}.
\]
Thus~\eqref{eq.lin.conv.2} follows.  It also follows that for 
$k=mK, \; m=1,2,\dots$
\[
f(x_k) - f^\star = f(x_{mK}) - f^\star \le \frac{L D_h(X^\star,x_{(m-1)K})}{2}  \le \frac{LD_h(X^\star,x_{0})}{2^{m-1}}  
\]
and thus~\eqref{eq.lin.conv.3} follows as well.
\end{proof}

To ease our exposition,  in Proposition~\ref{prop.lin.mirror}  and Proposition~\ref{prop.lin.mirror.2} we  assumed $L_k = L$ is known and used in Step 3 of Algorithm~\ref{algo.mirror}.  However, it is easy to see that these two results also hold if the assumption $L_k=L$ is relaxed to the assumption $L_k \le L$ and  $f(x_{k+1}) \le \min_{y\in X}\left\{f(x_k)+ \ip{\nabla f(x_k)}{y-x_k} + L_k D_h(y,x_k)\right\}$.  The latter condition is easier to implement via a standard backtracking procedure.   We also assume knowledge of suitable relative smoothness constants for the choice of stepsize $\alpha_k$ in Step 4 of Algorithm~\ref{algo.FW} and in Step 9 of Algorithm~\ref{algo.fwa} below.
As in Algorithm~\ref{algo.mirror}, this assumption can be relaxed via a standard backtracking procedure.

{  
\subsection{Frank-Wolfe algorithm}
Suppose $X\subseteq \R^n$ is a compact convex set and a {\em linear oracle} for $X$ is available, that is, the map
\[
g \mapsto \argmin_{y\in X}\ip{g}{x}
\]
is computable.

The Frank-Wolfe algorithm, also known as the conditional gradient algorithm, for \eqref{eq.Prob} is based on the following update for $x\in X:$
\begin{align*}
u &:= \argmin_{y\in X}\ip{\nabla f(x)}{y} \\
x_+&:= x + \alpha(u-x) \; \text{ for some } \; \alpha \in [0,1].
\end{align*}
Algorithm~\ref{algo.FW} gives a description of the Frank-Wolfe algorithm for \eqref{eq.Prob}.  

\begin{algorithm}
  \caption{Frank-Wolfe algorithm
    \label{algo.FW}}
  \begin{algorithmic}[1]
\State Pick $x_0 \in X$ ;
 \For{$k=0,1,2,\dots$}
 \State $u:=\argmin_{y\in X} \ip{\nabla f(x_k)}{y}$
\State $x_{k+1} = x_k + \alpha_k(u-x_k)$  \text{ for some } $\alpha_k \in [0,1]$ 
\EndFor
  \end{algorithmic}
\end{algorithm}

Let $\mathfrak{R}:=\frac{\mathfrak{r}^2}{2}$ where $\mathfrak{r}:X\times X \rightarrow \R_+$  is the {\em radial} distance defined as follows: for $x,y\in X$
\begin{equation}\label{eq.radial}
\mathfrak{r}(y,x):=
\inf \{\rho>0: y-x = \rho\cdot(u-x) \text{ for some } u\in X\}.
\end{equation}
Hence the relative  smoothness constant $L_{f,X,\mathfrak{R}}$ is the smallest $L >0$ such that for all $x,u\in X$ and $\alpha\in[0,1]$
\begin{equation}\label{eq.FW.relsmooth}
D_f(x+ \alpha(u-x),x)\le\frac{L\alpha^2}{2}.
\end{equation}
Observe that the relative smoothness constant $L_{f,X,\mathfrak{R}}$ is precisely the {\em curvature constant} of $f$ on $X$ defined by Jaggi~\cite{Jagg13}.  


\medskip

The relative quasi strong convexity constant $\mu_{f,X,\mathfrak{R}}^\star$ is 
the largest $\mu \ge 0$ such that
for all $x\in X$
\[
\frac{\mu \cdot \mathfrak{r}(\bar x,x)^2}{2} \le D_f(\bar x,x).
\]
Similarly, the relative functional growth constant $\mu_{f,X,\mathfrak{R}}^\sharp$ is the largest $\mu \ge 0$ such that
for all $x\in X$
\[
\frac{\mu \cdot \mathfrak{r}(\bar x,x)^2}{2} \le   f(x) - f^\star.
\]

The next result shows the linear convergence of Algorithm~\ref{algo.FW} when  
$L_{f,X,\mathfrak{R}}/\mu_{f,X,\mathfrak{R}}^\star$ or $L_{f,X,\mathfrak{R}}/\mu_{f,X,\mathfrak{R}}^\sharp$ is finite.
  As we note below,  Proposition~\ref{prop.lin.FW} is at least as sharp as the linear convergence rates established in~\cite{GuelM86,BeckT04}.

\begin{proposition}\label{prop.lin.FW} Suppose $L:=L_{f,X,\mathfrak{R}} < \infty$ and $\mu := \max\{\mu^\star_{f,X,\mathfrak{R}},\mu^\sharp_{f,X,\mathfrak{R}}/4\}>0.$ If each stepsize $\alpha_k$ in Step 4 of Algorithm~\ref{algo.FW} is chosen via
\[
\alpha_k = \argmin_{\alpha\in[0,1]}\left\{f(x) + \alpha \ip{\nabla f(x)}{u-x} + \frac{L\alpha^2}{2} \right\}
\]
then the iterates generated by Algorithm~\ref{algo.FW} satisfy
\[
f(x_k) - f^\star \le \left(1-\frac{\mu}{L} \right)^k(f(x_0) - f^\star).
\]
\end{proposition}
\begin{proof}
It suffices to show that at iteration $k$
\begin{equation}\label{eq.FW.claim}
\ip{\nabla f(x_k)}{x_k-u}^2 \ge 2\mu(f(x_k) - f^\star).
\end{equation}
Indeed, inequality~\eqref{eq.FW.relsmooth}, the choice of $\alpha_k$, and 
\eqref{eq.FW.claim}  imply that
\[
f(x_{k+1}) - f^\star \le f(x_k) - f^\star  - \frac{\ip{\nabla f(x_k)}{x_k-u}^2}{2L} \le \left(1-\frac{\mu}{L} \right)(f(x_k) - f^\star).
\]
We next show~\eqref{eq.FW.claim}.   The construction of the radial distance and the choice of $u$ in Algorithm~\ref{algo.FW} imply that
\[
\ip{\nabla f(x_k)}{x_k-\bar x_k} \le \mathfrak{r}(\bar x_k,x_k) \cdot \ip{\nabla f(x_k)}{x_k-u}.
\]
We next consider the two possible values of $\mu=\max\{\mu^\star_{f,X,\mathfrak{R}},\mu^\sharp_{f,X,\mathfrak{R}}/4\}$ separately.

\medskip

\noindent
{\bf Case 1: $\mu = \mu_{f,X,\mathfrak{R}}^\star$.}  In this case we have
\[
\frac{\mu\cdot \mathfrak{r}(\bar x_k, x_k)^2}{2} \le f^* - f(x_k) + \ip{\nabla f(x_k)}{x_k-\bar x_k} \le f^* - f(x_k) + 
\mathfrak{r}(\bar x_k, x_k) \ip{\nabla f(x_k)}{x_k-u}.
\]  
Rearranging and applying the arithmetic-mean geometric-mean inequality we get
\[
\ip{\nabla f(x_k)}{x_k-u} \ge \sqrt{2\mu(f(x_k) - f^\star)}.
\]
\medskip
\noindent
{\bf Case 2: $\mu = \mu_{f,X,\mathfrak{R}}^\sharp/4$.}  In this case we have
\[
2\mu\cdot\mathfrak{r}(\bar x_k,x_k)^2 \le f^* - f(x_k) \le \ip{\nabla f(x_k)}{x_k-\bar x_k} \le 
\mathfrak{r}(\bar x_k,x_k) \ip{\nabla f(x_k)}{x_k-u}.
\]  
Therefore the last term is at least as large as the geometric mean of the first two and we get
\[
\ip{\nabla f(x_k)}{x_k-u} \ge \sqrt{2\mu(f(x_k) - f^\star)}.
\]

\end{proof}

To conclude this subsection, we discuss some natural bounds on $L_{f,X,\mathfrak{R}}$ and  $\mu_{f,X,\mathfrak{R}}^\star$.  Recall that $\relint(X)$ denotes the relative interior of $X$.  Similarly, let $\relbdy(X)$ denote the relative boundary of $X$.
As it was previously discussed in~\cite{Jagg13}, from 
\eqref{eq.FW.relsmooth} it readily follows that if $f$ is $L_f$-smooth on $X$ for some norm $\|\cdot \|$ in $\R^n$ then
\[
L_{f,X,\mathfrak{R}} \le L_f \cdot \max_{u,x\in X}\|u-x\|^2 = L_f \cdot \diam(X)^2.
\]
On the other hand, if $f$ is $\mu_f$-strongly convex on $X$ for some norm $\|\cdot\|$ in $\R^n$ and the single element $x^\star \in X^\star$  satisfies $x^\star \in \relint(X)$ then for all $x\in X$ we have $\|x^\star-x\| = \mathfrak{r}(x^\star,x) \|u-x\| \ge \mathfrak{r}(x^\star,x)\|u-x^\star\|$ for some $u\in \relbdy(X)$.  The strong convexity of $f$ thus implies both 
\[
\mu_{f,X,\mathfrak{R}}^\star \ge 
\mu_f \cdot \dist(x^\star,\relbdy(X))^2 \; \text{ and } \;
\mu_{f,X,\mathfrak{R}}^\sharp \ge 
\mu_f \cdot \dist(x^\star,\relbdy(X))^2.
\]
Therefore when $f$ is both $L_f$-smooth and $\mu_f$-strongly convex and $x^\star = \argmin_{x\in X} f(x) \in \relint(X)$ we have
\[
\frac{L_{f,X,\mathfrak{R}}}{\mu_{f,X,\mathfrak{R}}^\star} \le \frac{L_f}{\mu_f} \cdot \left(\frac{\diam(X)}{\dist(x^\star,\relbdy(X))}\right)^2 \; \text{ and } \;
\frac{L_{f,X,\mathfrak{R}}}{\mu_{f,X,\mathfrak{R}}^\sharp} \le \frac{L_f}{\mu_f} \cdot \left(\frac{\diam(X)}{\dist(x^\star,\relbdy(X))}\right)^2.
\]
Observe that the right-hand side in both inequalities is an interesting combination of the usual condition number of $f$ and a kind of condition number of the set $X$ around the point $x^\star$.  The first bound above and Proposition~\ref{prop.lin.FW} yield a linear convergence result similar to~\cite[Theorem 2]{GuelM86} but with a sharper rate.

The above bounds can be extended to a broader context.  Suppose  $f = g \circ A$ for some strongly convex function $g:\R^m\rightarrow \R\cup\{\infty\}$ and $A\in \R^{m\times n}$.  Then for all $x\in X, x^\star \in X^\star$ we have 
\[
\|A(x^\star-x)\| \ge \mathfrak{r}(x^\star,x) \cdot \dist(Ax^\star,\relbdy(A(X))) \ge \mathfrak{r}(\bar x,x) \cdot \dist(Ax^\star,\relbdy(A(X))).\]
Consequently, if $X^\star\cap \relint(X)\ne \emptyset$ then for all $x^\star \in X^\star\cap \relint(X)$
\[
\mu_{f,X,\mathfrak{R}}^\star \ge 
\mu_g \cdot \dist(Ax^\star,\relbdy(A(X)))^2 \; \text{ and } \;
\mu_{f,X,\mathfrak{R}}^\sharp \ge 
\mu_g \cdot \dist(Ax^\star,\relbdy(A(X)))^2.
\]
Observe that $\dist(Ax^\star,\relbdy(A(X)))$ can in turn be bounded below as follows
\[
\dist(Ax^\star,\relbdy(A(X))) \ge \frac{1}{\|(A|\lspan(X-X))^{-1}\|} \cdot  \dist(x^\star,\relbdy(X)).
\] 
Therefore when $f = g\circ A$ where $g$ is $L_g$-smooth and $\mu_g$-strongly convex then for all $x^\star \in X^\star \cap \relint(X)$ both 
$L_{f,X,\mathfrak{R}}/\mu_{f,X,\mathfrak{R}}^\star$ and $L_{f,X,\mathfrak{R}}/\mu_{f,X,\mathfrak{R}}^\sharp$ are bounded above by
\[
 \frac{L_f}{\mu_f} \cdot \left(\frac{\diam(AX)\cdot \|(A|\lspan(X-X))^{-1}\|}{\dist(x^\star,\relbdy(X))}\right)^2.
\]
This bound and Proposition~\ref{prop.lin.FW} yield a linear convergence result similar to~\cite[Proposition 3.2]{BeckT04} but with a sharper rate. 
}

\subsection{Frank-Wolfe  with away steps algorithm}

Suppose $X\subseteq \R^n$ is a {\em polytope} and a {\em vertex linear oracle} for $X$ is available, that is, the map
\[
g \mapsto \argmin_{y\in X}\ip{g}{x}
\]
is computable and outputs a vertex of $X$ for all $g\in \R^n$.

For this kind of linear oracle, each step of the Frank-Wolfe algorithm {\em adds weight} to some vertex $u$.  The basic idea of the Frank-Wolfe  with away steps algorithm is to combine {\em regular steps} of the Frank-Wolfe algorithm with {\em away steps} that {\em reduce weight} from some vertex $a$.  To that end, the algorithm requires an additional  {\em vertex representation} of $x\in X$.  More precisely, let $S(x)\subseteq \vertices(X)$ and $\lambda(x) \in \Delta(S(x)) := \{z\in \R^{S(x)}_+: \|z\|_1 = 1\}$ be such that
\[
x = \sum_{s\in S(x)} \lambda_s(x) s \; \text{ and } \;  \lambda(x) > 0.
\]
Algorithm~\ref{algo.fwa} describes a Frank-Wolfe with away steps algorithm.
We should highlight that although the set $\vertices(X)$ could be immense, the algorithm does not require it explicitly.  Instead the algorithm only maintains $S(x)$ and $\lambda(x)$ that are far more manageable.  Indeed, by using the IRR procedure in~\cite{BeckS15} or its modification described in~\cite{Gutm19}, Step 10 in Algorithm~\ref{algo.fwa} can guarantee that the sets $S(x_k)$ have size at most $n+1$ for $k=0,1,\dots$.

\begin{algorithm}
  \caption{Frank-Wolfe with away steps algorithm
    \label{algo.fwa}}
  \begin{algorithmic}[1]
\State Pick $x_0 \in \vertices(X) ; \; S(x_0) := \{x_0\} ; \lambda(x_0) = 1$
 \For{$k=0,1,2,\dots$}
 \State $u:=\argmin_{y\in X} \ip{\nabla f(x_k)}{y} ; \;\; a:= \argmax_{y\in S(x_k)} \ip{\nabla f(x_k)}{y} $

 \If{$\ip{\nabla f(x_k)}{u-x_k} < \ip{\nabla f(x_k)}{x_k-a}$  } (regular step)
\State $v:=u - x_k ; \; \alpha_{\max} = 1 ;$  
\quad \Else 
$\;\;$ (away step)
\State  $v:= x_k - a ; \; \alpha_{\max} = \frac{\lambda_a(x_k)}{1-\lambda_a(x_k)} ;$ 
\EndIf 
\State $x_{k+1} := x_k + \alpha_k v$ for some  $\alpha_k \in [0,\alpha_{\max}]$ 
\State update $S(x_{k+1})$ and $\lambda(x_{k+1})$
\EndFor
  \end{algorithmic}
\end{algorithm}

Proposition~\ref{prop.lin.fwa} below establishes the linear convergence of Algorithm~\ref{algo.fwa} under suitable  relative smoothness and quasi strong convexity or functional growth conditions.  To that end, we consider two variants of the radial distance.   Let $\mathfrak{D}:=\frac{\mathfrak{d}^2}{2}$ where $\mathfrak{d}:X\times X\rightarrow\R_+$ is the {\em diametral distance} defined via
\begin{equation}\label{eq.diametral}
\mathfrak{d}(y,x):=\inf\{\delta>0: y-x= \delta\cdot(u-w) \text{ for some } u-w\in X\}.
\end{equation}
The relative smoothness constant $L_{f,X,\mathfrak{D}}$ is the smallest $L >0$ such that for all $x,u,w\in X$ and $\alpha\in[0,1]$ with $x+ \alpha(u-w)\in X$
\begin{equation}\label{eq.FWA.relsmooth}
D_f(x+ \alpha(u-w),x) \le \frac{L\alpha^2}{2}.
\end{equation}
The relative smoothness constant $L_{f,X,\mathfrak{D}}$ is precisely the {\em away curvature constant} of $f$ on $X$ defined by Lacoste-Julien and Jaggi~\cite{LacoJ15}.

\bigskip

To capture the appropriate relative strong convexity conditions, we rely on a more involved variant of the radial distance.  For $x\in X$, let $\mathbf{S}(x)$ denote the collection of all subsets $S(x) \subseteq \vertices(X)$ such that $x$ is a {\em positive} convex combination of the elements in $S(x)$.  Let $\mathfrak{G}:=\frac{\mathfrak{g}^2}{2}$ where $\mathfrak{g}:X\times X\rightarrow\R_+$ is defined via
\begin{equation}\label{eq.gradiental}
\mathfrak{g}(y,x):=\inf\left\{\gamma>0: \ip{\nabla f(x)}{x-y} \le \gamma \cdot \min_{S(x) \in \mathbf{S}(x)} \max_{a\in S(x),u\in X} \ip{\nabla f(x)}{a-u}\right\}.
\end{equation}
The relative strong convexity constant $\mu_{f,X,\mathfrak{G}}$ is at least as large as 
\[
\sup\left\{\mu :\frac{\mu\cdot \mathfrak{g}(y,x)^2}{2} \le D_f(y,x) \text{ for all } x,y\in X\right\}.
\]
The latter quantity 
is precisely the {\em geometric strong convexity constant} defined by Lacoste-Julien and Jaggi~\cite[Appendix C]{LacoJ15}.  Notice that it matches $\mu_{f,X,\mathfrak{G}}$ when $f$ is strictly convex because in that case $Z_{f,X}(y) = \{y\}$ for all $y\in X$.  Otherwise, $\mu_{f,X,\mathfrak{G}}$ could be larger.

The relative quasi strong convexity constant $\mu_{f,X,\mathfrak{G}}^\star$ is the largest $\mu \ge 0$ such that
for all $x\in X$
\[
\frac{\mu \cdot\mathfrak{g}(\bar x,x)^2}{2} \le D_f(\bar x,x).
\]
Similarly, the relative functional growth constant $\mu_{f,X,\mathfrak{G}}^\sharp$ is the largest $\mu \ge 0$ such that
for all $x\in X$
\[
\frac{\mu \cdot\mathfrak{g}(\bar x,x)^2}{2} \le   f(x) - f^\star.
\]
Since $\mu_{f,X,\mathfrak{G}} \le \mu_{f,X,\mathfrak{G}}^\star$ and $\mu_{f,X,\mathfrak{G}}$ is at least as large as the geometric strong convexity constant in~\cite[Appendix C]{LacoJ15}, the following linear convergence result is at least as sharp as the one given in~\cite[Theorem 8]{LacoJ15} for the Frank-Wolfe  with away steps algorithm.

\begin{proposition}\label{prop.lin.fwa} Suppose $L:=L_{f,X,\mathfrak{D}} < \infty$ and $\mu := \max\{\mu^\star_{f,X,\mathfrak{G}},\mu^\sharp_{f,X,\mathfrak{G}}/4\}>0.$ If each stepsize $\alpha_k$ in Step 9 of Algorithm~\ref{algo.fwa} is chosen via
\[
\alpha_k = \argmin_{\alpha \in [0,\alpha_{\max}]}\left\{f(x) + \alpha \ip{\nabla f(x)}{u-x} + \frac{L\alpha^2}{2} \right\}
\]
then the iterates generated by Algorithm~\ref{algo.fwa} satisfy
\begin{equation}\label{eq.lin.fwa}
f(x_k) - f^\star \le \left(1-\min\left\{\frac{1}{2},\frac{\mu}{4L}\right\} \right)^{k/2}(f(x_0) - f^\star).
\end{equation}
\end{proposition}
\begin{proof}
This proof follows a similar reasoning to the proof of Proposition~\ref{prop.lin.FW}. First we claim that at iteration $k$
\begin{equation}\label{eq.FWA.claim}
\ip{\nabla f(x_k)}{a-u}^2 \ge 2\mu(f(x_k) - f^\star).
\end{equation}
To show this claim, consider the two possible values of $\mu := \max\{\mu^\star_{f,X,\mathfrak{G}},\mu^\sharp_{f,X,\mathfrak{G}}/4\}$ separately.

\medskip

\noindent
{\bf Case 1: $\mu = \mu^\star_{f,X,\mathfrak{G}}$.}  In this case we have
\[
\frac{\mu\cdot\mathfrak{g}(\bar x_k,x_k)^2}{2} \le f^* - f(x_k) + \ip{\nabla f(x_k)}{x_k-\bar x_k} \le f^* - f(x_k) +
\mathfrak{g}(\bar x_k,x_k) \ip{\nabla f(x_k)}{a-u}.
\]  
Rearranging and applying the arithmetic-mean geometric-mean inequality we get
\[
\ip{\nabla f(x_k)}{a-u} \ge \sqrt{2\mu(f(x_k) - f^\star)}.
\]
\medskip
\noindent
{\bf Case 2: $\mu = \mu_{f,X,\mathfrak{G}}^\sharp/4$.}  In this case we have
\[
2\mu \cdot \mathfrak{g}(\bar x_k,x_k)^2 \le f^* - f(x_k) \le \ip{\nabla f(x_k)}{x_k-\bar x_k} \le 
\mathfrak{g}(\bar x_k,x_k) \ip{\nabla f(x_k)}{a-u}.
\]  
Therefore the last term is at least as large as the geometric mean of the first two and we get
\[
\ip{\nabla f(x_k)}{a-u} \ge \sqrt{2\mu(f(x_k) - f^\star)}.
\]
To finish the proof, we next show 
\eqref{eq.lin.fwa} by relying on~\eqref{eq.FWA.claim}.  To do so, we replicate some of the main ideas previously introduced in~\cite{BeckS15,LacoJ15,PenaR16}.

 The choice of $v$ at iteration $k$ and~\eqref{eq.lin.fwa} imply that
\begin{equation}\label{eq.step.size}
\ip{\nabla f(x_k)}{v}^2 \ge \frac{\ip{\nabla f(x_k)}{a-u}^2}{4} \ge \frac{\mu(f(x_k)-f^\star)}{2}.
\end{equation} 
We consider separately the three possible cases that can occur for $\alpha_k$ at iteration $k$, namely $\alpha_k < \alpha_{\max}$, $\alpha_k = \alpha_{\max} \ge 1,$ and $\alpha_k = \alpha_{\max} < 1.$

\medskip

\noindent
{\bf Case 1:} $\alpha_k < \alpha_{\max}$. In this case $|S(x_{k+1})| \le |S(x_k)|+1$.  In addition, inequalities~\eqref{eq.FWA.relsmooth} and~\eqref{eq.step.size}, and the choice of $\alpha_k$ imply that
\begin{equation}\label{eq.dec}
f(x_{k+1}) - f(x_k) \le -\frac{\ip{\nabla f(x_k)}{v}^2}{2L} \le
-\frac{\ip{\nabla f(x_k)}{a-u}^2}{8L}
\le -\frac{\mu}{4L}(f(x_k) - f^\star).
\end{equation}

\medskip

\noindent
{\bf Case 2:} $\alpha_k = \alpha_{\max} \ge 1$.  In this case $|S(x_{k+1})| \le |S(x_k)|$.  In addition, inequality~\eqref{eq.FWA.relsmooth}, the choice of $v$, and the convexity of $f$ imply that
\begin{equation}\label{eq.dec.2}
f(x_{k+1}) - f(x_k) \le \frac{1}{2}\ip{\nabla f(x_k)}{v} \le \frac{1}{2}\ip{\nabla f(x_k)}{\bar x_k - x_k} \le -\frac{1}{2}(f(x_k) - f^\star).
\end{equation}

\medskip

\noindent
{\bf Case 3:} $\alpha_k = \alpha_{\max} < 1$.  In this case $|S(x_{k+1})| \le |S(x_k)|-1$.  In addition,~\eqref{eq.FWA.relsmooth} and the choice of $\alpha_k$ imply that
\[
f(x_{k+1}) - f(x_k) \le 0.
\]
We next show that in the first $k$ iterations Case 3 can occur at most $k/2$ times by using the  argument introduced by Lacoste-Julien and Jaggi in~\cite{LacoJ15}.  Since $|S(x_0)| = 1$ and $|S(x_i)| \ge 1$ for $i=1,2,\dots,$ it follows that for each iteration when Case 3 occurred
there must have been at least one previous iteration when Case 1 occurred. Hence in the first $k$ iterations Case 3 could occur at most $k/2$ times.  

To finish the proof, observe that at every iteration $k$ when Case 1 or Case 2 occur inequalities~\eqref{eq.dec} and~\eqref{eq.dec.2} yield
\[
f(x_{k+1}) - f^\star  = f(x_k) - f^\star + f(x_{k+1}) - f(x_k) \le \left(1-\min\left\{\frac{1}{2},\frac{\mu}{4L}\right\} \right)(f(x_{k}) - f^\star).
\]
We note that the minimum in the last expression is 
is necessary because $\mu_{f,X,\mathfrak{G}}^\sharp > 2L_{f,X,\mathfrak{D}}$ may indeed occur. For a concrete example, see~\cite[Example 6]{PenaR16}.

\end{proof}

We next discuss some bounds on $L_{f,X,\mathfrak{D}}$ and on $\mu_{f,X,\mathfrak{G}}, \mu_{f,X,\mathfrak{G}}^\star, \mu_{f,X,\mathfrak{G}}^\sharp$  in terms of the set $A:=\vertices(X)$.  We should note that the bounds below on $L_{f,X,\mathfrak{D}}$ and on $\mu_{f,X,\mathfrak{G}}$ have also been derived, albeit following a different approach, in~\cite[Appendix C]{LacoJ15}.

From~\eqref{eq.FWA.relsmooth} it readily follows that if $f$ is $L_f$-smooth on $X$ for some norm $\|\cdot \|$ in $\R^n$ then
\[
L_{f,X,\mathfrak{D}} \le L_f \cdot \max_{x,y\in X}\|x-y\|^2 = L_f \cdot \diam(X)^2 = L_f \cdot \diam(A)^2.
\]
On the other hand, from~\cite[Theorem 1]{PenaR16} it follows that for all $x,y\in X$
\[
\|y-x\| \ge  \mathfrak{g}(y,x) \cdot \Phi(A) 
\]
where $\Phi(A) = \dmin_{F\in\faces(\conv(A))\atop \emptyset\ne F\ne \conv(A)}
\dist(F,\conv(A\setminus F))$.
 
\medskip 
 
Hence if $f$ is $\mu_f$-strongly convex on $X$ for some norm $\|\cdot \|$ in $\R^n$ then for all $y,x\in X$ we have
\[
\frac{\mu_f \Phi(A)^2 \mathfrak{g}(y,x)^2}{2} \le \frac{\mu_f \|y-x\|^2}{2} \le D_f(y,x)
\]
and consequently
$$\mu_{f,X,\mathfrak{G}} \ge \mu_f \cdot \Phi(A)^2.$$ 

Therefore when $f$ is both $L_f$-smooth and $\mu_f$-strongly convex on $X$ for some norm $\|\cdot\|$ in $\R^n$ we have
\[
\frac{L_{f,X,\mathfrak{D}}}{\mu_{f,X,\mathfrak{G}}} \le \frac{L_f}{\mu_f} \cdot \left(\frac{\diam(A)}{\Phi(A)}\right)^2.
\]
Once again, the right-hand side is an interesting combination of the usual condition number of $f$ and a kind of condition number of $A=\vertices(X)$.
Furthermore, by proceeding as in Example~\ref{ex.simplex} it follows that when $f$ is of the form $f(x) = \frac{1}{2}\|Bx-b\|_2^2$ for some $B\in \R^{m\times n}$ and $b\in \R^m$ we have $ L_{f,X,\mathfrak{D}} = \diam(BA)^2$ and $\mu_{f,X,\mathfrak{G}} = \Phi(BA)^2$.  Thus for $f(x) = \frac{1}{2}\|Bx-b\|_2^2$ we have
\[
\frac{L_{f,X,\mathfrak{D}}}{\mu_{f,X,\mathfrak{G}}} = \left(\frac{\diam(BA)}{\Phi(BA)}\right)^2.
\]
This illustrates how the condition number of $f$ relative to $X$ depends on how the shape of $X$ and $f$ fit together.

We also have the following sharper lower bound on $\mu_{f,X,\mathfrak{G}}^\star$.  From~\cite[Theorem 3]{PenaR16} it follows that 
\[
\|x^\star-x\| \ge \mathfrak{g}(x^\star,x) \cdot 
\min_{F\in\faces(G)\atop \emptyset\ne F\ne \conv(A)}
\dist(F,\conv(A\setminus F)) 
\]
where $G\in \faces(\conv(A))$ is the smallest face of $\conv(A)=X$ that contains $X^\star$.  It thus follows that if $f$ is $\mu_f$-strongly convex on $X$ for some norm $\|\cdot\|$ then 
\[
\mu_{f,X,\mathfrak{G}}^\star \ge  \mu_f\cdot
\min_{F\in\faces(G)\atop \emptyset\ne F\ne \conv(A)}
\dist(F,\conv(A\setminus F))^2. 
\]
Finally we note that Theorem~\ref{thm.main.growth}  implies that $\mu_{f,X,\mathfrak{G}}^\sharp > 0$ when $f$ is of the form $f(x) = g(Ex)+\ip{b}{x}$ for some strongly convex function $g$.  Indeed, with a slight abuse of notation, let $A\in \R^{n\times N}$ denote the matrix whose columns are the elements of $A$ and consider the function $\tilde f: \R^N \rightarrow \R$ defined via $\tilde f := f\circ A$. Observe that for $u,v\in \Delta_{N-1}$ 
\[
D_f(Av,Au) = D_{\tilde f}(v,u) \text{ and } \mathfrak{g}(Au,Av) \le \frac{\|u-v\|_1}{2}.
\]
Consequently,
\[
\mu_{f,X,\mathfrak{G}}^\sharp \ge 4\mu_{\tilde f,\Delta_{N-1},D}^\sharp
\]
for the distance function $D(v,u):= \frac{1}{2}\|v-u\|_1^2$. 
The functional growth constant $\mu_{\tilde f,\Delta_{N-1},D}^\sharp$ in turn can be bounded below as detailed in  Theorem~\ref{thm.main.growth} since $\tilde f$ can be written as $\tilde f(u) = g(EAu) + \ip{b}{Au}$ and $g$ is strongly convex.

\medskip

The linear convergence bounds in Proposition~\ref{prop.lin.fwa} are tight modulo some small constants.  This can be readily inferred from~\cite[Example 3 and Example 4]{PenaR16}.

\bibliographystyle{plain}

\appendix
{  \section{Proof of Proposition~\ref{prop.gral.local}}\label{appendix}

The construction of $T_X(x;A,S)$ implies
$T_X(x;A,S) \subseteq T_X(x)$ and~$\|(A|T_X(x;A,S))^{-1}\| \le \|(A|T_X(x))^{-1}\|$~for all $x\in X$.  Hence\[\sup_{C\in\T(A|X,S)} \|(A|C)^{-1}\| \le \max_{C\in\T(X)} \|(A|C)^{-1}\| = \max_{C\in\T(A|X)} \|(A|C)^{-1}\|\]  
where the last step follows from~\cite[Lemma 1]{PenaVZ18}.  This proves the second inequality in~\eqref{eq.sharp}.

Let $H:=\sup_{C\in\T(A|X,S)} \|(A|C)^{-1}\|.$ The first inequality in~\eqref{eq.sharp} can be stated as follows: for all
$y\in S$ and $x\in X$
\begin{equation}\label{eq.Hoffman}
\|Z_{A,X}(y)-x\| \le H \cdot \|Ay-Ax\|.
\end{equation}
We prove~\eqref{eq.Hoffman} by contradiction.  Suppose that there exist $y\in S$ and $x\in X\setminus Z_{A,X}(y)$ such that  $\|Z_{A,X}(y)-x\| > H \cdot \|Ay-Ax\|.$ That is, 
\begin{equation}\label{eq.contra}
A\tilde y = Ay, y\in X \Rightarrow \|\tilde y - x\| > H \cdot \|Ay-Ax\|.
\end{equation}
Let $v:=(Ay-Ax)/\|Ay-Ax\|$ and consider the convex optimization problem
\begin{equation}\label{eq.opt.prob}
\begin{array}{rl}
\displaystyle\max_{u,t} & t \\
& A u = t v \\
& x + u \in X \\
& \|u\| \le H \cdot t.
\end{array}
\end{equation}
Observe that $v \in A(T_X(x;A,S))$ since $y - x \in T_X(x;A,S)$.  Thus there exists $u\in T_X(x;A,S)$ such that $Au = v$ and $$\|u\| \le \|(A|T_X(x;A,S))^{-1}\| \le H.$$ Therefore there exists $(u,t)$ feasible for~\eqref{eq.opt.prob} with $t > 0$.  On the other hand,~\eqref{eq.contra} implies that there does not exist any $(u,t)$ feasible for~\eqref{eq.opt.prob} with $t = \|Ay-Ax\|$.  It thus follows that~\eqref{eq.opt.prob} has an optimal solution $(\hat u, \hat t)$ with $0< \hat t < \|Ay-Ax\|$.  Now consider the  modification of~\eqref{eq.opt.prob} obtained by replacing $x$ with $x + \hat u \in X$:
\begin{equation}\label{eq.opt.prob.2}
\begin{array}{rl}
\displaystyle\max_{u,t} & t \\
& A u = t v \\
& x + \hat u + u \in X \\
& \|u\| \le H \cdot t.
\end{array}
\end{equation}
Proceeding as above with $x + \hat u$ in lieu of $x$ it follows that~\eqref{eq.opt.prob.2} has an optimal solution $(u',t')$ with $0<t'<\|Ay-Ax\| - \hat t$.  In particular, $(\hat u + u', \hat t + t')$ is a feasible solution to~\eqref{eq.opt.prob} with $\hat t + t' > \hat t$ which contradicts the optimality of $(\hat u,\hat t).$  We therefore conclude that~\eqref{eq.Hoffman} must hold and thus~\eqref{eq.sharp} is proven.

We next prove~\eqref{eq.prop.local} when $A(S)$ is  convex.  To that end, suppose $C \in \T(A|X,S)$ and $0 < \epsilon < \|(A|C)^{-1}\|$.  Then $C = T_X(\hat x;A,S)$ for some $\hat x\in X$.  Let $\hat v\in C$ be such that $A\hat v \ne 0$ and $\|v\| \ge  (\|(A|C)^{-1} \|-\epsilon)\cdot \|A\hat v\|$ for all $v\in C$ with $Av = A\hat v$.  By scaling $\hat v$ if necessary we can assume that $A(\hat x + \hat v) \in {\conv}(A(S)) = A(S)$ and thus $A(\hat x + \hat v) = A\hat y$ for some $\hat y\in S$.  Observe that $\hat x + v \in Z_{A,X}(\hat y)$ implies both $v\in C$ and $Av = A\hat v$. It thus follows that
\[
\frac{1}{\|(A| C)^{-1}\|-\epsilon} \ge 
\frac{\|A\hat v\|}{\|Z_{A,X}(\hat y) - \hat x\|}
=
\frac{\|A\hat y - A\hat x\|}{\|Z_{A,X}(\hat y) - \hat x\|} \ge \inf_{y\in S,\, x\in X \atop x\not\in   Z_{A,X}(y)} \frac{\|Ay-Ax\|}{\|Z_{A,X}(y)-x\|}.
\]
Since this holds for all $C \in \T(A|X,S)$ and $0<\epsilon<\|(A|C)^{-1}\|$  identity~\eqref{eq.prop.local} follows.
\qed
}

\end{document}